\documentclass[12pt,leqno]{article}
\usepackage{amssymb,amsfonts,amsmath,amsthm,amscd,mathrsfs}
\usepackage{graphicx}
\def\IE{{\mathbb E}}
\def\IP{{\mathbb P}}
\def\IR{{\mathbb R}}

\def\IZ{{\mathbb Z}}

\def\n{\noindent}
\def\dsl{\textstyle\sum\limits}
\def\dil{\textstyle\int}
\def\pr{\textstyle\prod\limits}
\def\dis{\displaystyle}

\def\ov{\overline}

\def\f{\footnotesize}
\def\r{\rightarrow}

\def\wt{\widetilde}

\def\cC{{\cal C}}

\def\cL{{\cal L}}

\def\cP{{\cal P}}
\def\cF{{\cal F}}

\def\cV{{\cal V}}

\def\cZ{{\cal Z}}

\dimendef\dimen=0

\newtheorem{theorem}{Theorem}[section]
\newtheorem{lemma}[theorem]{Lemma}
\newtheorem{corollary}[theorem]{Corollary}
\newtheorem{proposition}[theorem]{Proposition}
\newtheorem{remark}[theorem]{Remark}

\begin{document}

\noindent
~

\bigskip
\begin{center}
{\bf LEVEL-SET PERCOLATION FOR THE GAUSSIAN FREE FIELD \\
ON A TRANSIENT TREE}
\end{center}

\begin{center}
Angelo Ab\"acherli and Alain-Sol Sznitman
\end{center}

\bigskip\bigskip
\begin{abstract}
We investigate level-set percolation of the Gaussian free field on transient trees, for instance on super-critical Galton-Watson trees conditioned on non-extinction. Recently developed Dynkin-type isomorphism theorems provide a comparison with percolation of the vacant set of random interlacements, which is more tractable in the case of trees. If $h_*$ and $u_*$ denote the respective (non-negative) critical values of level-set percolation of the Gaussian free field and of the vacant set of random interlacements, we show here that $h_* < \sqrt{2u}_*$ in a broad enough set-up, but provide an example where $0 = h_* = u_*$ occurs. We also obtain some sufficient conditions ensuring that $h_* > 0$.

\medskip
\begin{center}
{\bf R\'esum\'e}
\end{center}

Nous \'etudions la percolation de niveau pour le champ libre gaussien sur des arbres transients,
par exemple sur des arbres de Galton-Watson surcritiques conditionn\'es \`a survivre. Des th\'eor\`emes
de type isomorphisme de Dynkin r\'ecemment obtenus offrent un outil de comparaison avec la percolation
de l\textsc{\char13}ensemble vacant pour les entrelacs al\'eatoires, qui se trouve \^etre plus simple \`a \'etudier dans le cas des arbres. Si $h_*$ et $u_*$ d\'esignent les valeurs critiques respectives de la percolation de niveau du champ libre gaussien, et de l\textsc{\char13}ensemble vacant des entrelacs al\'eatoires, nous montrons dans un cadre assez g\'en\'eral que $h_* < \sqrt{2u}_*$, mais pr\'esentons un exemple pour lequel on a les \'egalit\'es $0 = h_* = u_*$. Nous obtenons aussi des conditions suffisantes qui impliquent que $h_* > 0$.

\end{abstract}

\vfill 

\n
Departement Mathematik 
\\
ETH Z\"urich\\
CH-8092 Z\"urich\\
Switzerland

\vfill

~
\newpage
\thispagestyle{empty}
~

\newpage
\setcounter{page}{1}

\setcounter{section}{-1}
\section{Introduction}

In this work we investigate level-set percolation of the Gaussian free field on a transient tree. Recently, over the last couple of years, various versions of Dynkin-type isomorphism theorems have related Gaussian free fields to random interlacements, see for instance \cite{LupuA}, \cite{Rose14}, \cite{SaboTarr}, \cite{Szni12b}, \cite{Szni}, \cite{Zhai}. They have fostered an interplay between level-set percolation of the Gaussian free field and percolation of the vacant set of random interlacements. In the case of transient trees, the vacant cluster of random interlacements at a given site can be expressed in terms of Bernoulli site percolation, see \cite{Teix09b}. This makes percolation of the vacant set of random interlacements more tractable, in particular with the help of the methods developed in \cite{Lyon90}, \cite{Lyon92}, \cite{LyonPere16}. In view of the above mentioned interplay, this feature raises the hope of gaining further insight into the more intricate level-set percolation of the Gaussian free field on a transient tree. This strategy was implemented in \cite{Szni} in the case of $(d+1)$-regular trees when $d \ge 2$. In particular, it was shown there that $0 < h_* < \sqrt{2 u}_*$, if $h_*$ and $u_*$ stand for the respective critical values of level-set percolation of the Gaussian free field and of percolation of the vacant set of random interlacements. Here, we resume this approach in the broader context of transient trees, in particular for super-critical Galton-Watson trees conditioned on non-extinction. Whereas we provide an example showing that $0 = h_* = \sqrt{2u_*}$ may occur, we prove under rather general assumptions that $h_* < \sqrt{2u_*}$, and derive sufficient conditions ensuring that $h_* > 0$.

\medskip
Let us now describe the set-up and our results in more detail. We consider a locally finite tree (that is, a locally finite connected graph without loops) with vertex set $T$, such that each edge has unit weight and the corresponding weighted graph is transient. The discrete time random walk on $T$, when located in $x$, jumps to any given neighbor with probability ${\rm deg}(x)^{-1}$, where ${\rm deg}(x)$ stands for the degree of $x$. We write $P_x$ for the canonical law of the walk starting in $x$, $E_x$ for the corresponding expectation and $(X_k)_{k \ge 0}$ for the walk. The Green function is symmetric, positive, and equals
\begin{equation}\label{0.1}
g(x,y) = \mbox{\small $\dis\frac{1}{{\rm deg}(y)} $}\;E_x \Big[\dsl_{k=0}^\infty 1\{X_k = y\}\Big], \;\mbox{for $x,y \in T$}.
\end{equation}

\n
We write $\IP^G$ for the canonical law on $\IR^T$ of the Gaussian free field on $T$, and denote by $(\varphi_x)_{x \in T}$ the canonical field, so that under $\IP^G$
\begin{equation}\label{0.2}
\mbox{$(\varphi_x)_{x \in T}$ is a centered Gaussian field with covariance $g(\cdot,\cdot)$}.
\end{equation}

\n
The critical value of the level-set percolation of $\varphi$ is defined as 
\begin{equation}\label{0.3}
\mbox{$h_* = \inf\{ h \in \IR$; $\IP^G$-a.s., all connected components of $\{\varphi \ge h\}$ are finite$\}$}
\end{equation}
(here $\{\varphi \ge h\} = \{x \in T; \varphi_x \ge h\}$ and $\inf \phi = \infty$).

\medskip
By a general argument of \cite{BricLeboMaes87}, recalled in the Appendix, one knows that
\begin{equation}\label{0.4}
0 \le h_* \le \infty .
\end{equation}

\n
Further, given $u \ge 0$, we consider the vacant set $\cV^u$ of random interlacements at level $u$. This random subset of $T$ is governed by a probability $\IP^I$ (see (\ref{1.32})) and $\cV^u$ becomes thinner as $u$ increases. The critical value for the percolation of $\cV^u$ is defined as
\begin{equation}\label{0.5}
\mbox{$u_* = \inf\{u \ge 0$; $\IP^I$-a.s., all connected components of $\cV^u$ are finite$\}$}.
\end{equation}

\n
To describe our results we introduce some base point $x_0$ of $T$, and define for any $x$ in $T$, the sub-tree $T_x$ of descendents of $x$, consisting of those $y$ in $T$ for which the geodesic path between $x_0$ and $y$ goes through $x$ (see the beginning of Section 1). We then write, see~(\ref{1.4}),
\begin{equation}\label{0.6}
\mbox{$R^\infty_x =$ the effective resistance between $x$ and $\infty$  in $T_x$}.
\end{equation}

As an application of the cable graph methods initiated in \cite{LupuA}, we show in Corollary \ref{cor2.3} of Section 2 that when
\begin{equation}\label{0.7}
\begin{array}{l}
\mbox{for some $A > 0$, the (deterministic) set $\{x \in T$; $R^\infty_x > A\}$} 
\\
\mbox{only has finite components,}
\end{array}
\end{equation}
then one has
\begin{equation}\label{0.8}
0 \le h_* \le \sqrt{2 u}_*. 
\end{equation}

\n
We also present in Remark \ref{rem2.4} 2) an example where
\begin{equation}\label{0.9}
0 = h_* = \sqrt{2u_*}.
\end{equation}

\n
As an aside one may wonder whether $\IP^G [x_0 \stackrel{\varphi \ge 0}{\longleftrightarrow} \infty] >0$ holds under (\ref{0.7}). This issue is linked to the geometry of the sign-clusters of the Gaussian free field on the cable system, see Remark \ref{rem2.4} 3).

\medskip
Getting hold of strict inequalities strengthening (\ref{0.8}) is more delicate. We provide in Theorem \ref{theo3.3} a rather general sufficient condition, which ensures that $h_* < \sqrt{2u_*}$. This result comes as an application of the special coupling between random interlacements and the Gaussian free field, which appears in Corollary \ref{cor2.3} and was constructed in \cite{Szni} as a refinement of \cite{LupuA}. More precisely, we show in Theorem \ref{theo3.3} that when $0 < u_* < \infty$, and conditions (\ref{3.1}) and (\ref{3.2}) hold, that is, for some $A,B,M, \delta > 0$, for all distant vertices $x$ in $T$ having an infinite line of descent
\begin{align}
& \dsl_{y \in (x_0,x)} 1\{R_y^\infty \le A, \;d_{y^-} \le M\} \ge \delta |x|, \label{0.10}
\\[1ex]
& \dsl_{y \in (x_0,x]} \;\mbox{\small $\dis\frac{1}{R^\infty_y (1 + R^\infty_y)}$}\; \le B |x|  \label{0.11}
\end{align}

\n
(with $|x|$ the distance of $x$ to $x_0$, and $y^-$ the parent of $y$, see the beginning of Section 1 for notation), then, one has
\begin{equation}\label{0.12}
h_* < \sqrt{2u_*} .
\end{equation}

\n
Importantly, to take advantage of the above mentioned special coupling, we show in Proposition \ref{prop3.2} that (\ref{0.10}) implies an exponential decay in $|x|$ of $\IP^G [x_0 \stackrel{\varphi \ge 0}{\longleftrightarrow} x]$. The proof is based on an idea of ``entropic repulsion'' in the spirit of \cite{Giac03}, p.~13,~14.

\medskip
In Proposition \ref{prop4.2} we give a sufficient condition for $h_* > 0$, namely the existence of an infinite binary sub-tree of sites having uniformly bounded degree. As an application of Theorem \ref{theo3.3} and Proposition \ref{prop4.2}, we see for instance that
\begin{align}
&\mbox{$0 < h_* < \sqrt{2u_*} < \infty$, when $T$ has bounded degree and outside a} \label{0.13}
\\
&\mbox{finite subset of $T$, each site has degree at least $3$}. \nonumber
\end{align}

\n
Incidentally, in the case of $\IZ^d$, $d \ge 3$, the inequality $h_* \le \sqrt{2u}_*$ is known, see \cite{LupuA}, but the strict inequality $h_* < \sqrt{2u}_*$ is presently open, and $h_* > 0$ is only known when $d$ is sufficiently large, see \cite{DrewRodr15}, \cite{RodrSzni13}.

\medskip
Our results also apply to typical realizations of super-critical Galton-Watson trees conditioned on non-extinction. In this case, one knows from \cite{Tass10} that $u_*$ is deterministic and $0 < u_* < \infty$. There is even a reasonably explicit formula characterizing $u_*$, which is recalled in (\ref{5.4}). We show in Lemma \ref{lem5.1} that $h_*$ is deterministic as well. In the more challenging Proposition \ref{prop5.2} we show that (\ref{3.1}) (or (\ref{0.10})) holds almost surely on non-extinction. In particular, as by-product, we deduce that
\begin{equation}\label{0.14}
\begin{array}{l}
0 \le h_* \le \sqrt{2u_*} < \infty, \;\mbox{and almost surely on non-extinction}
\\
\mbox{$\IP^G[x_0 \stackrel{\varphi \ge 0}{\longleftrightarrow} x]$ has exponential decay in $|x|$}.
\end{array}
\end{equation}

\n
When the offspring distribution has in addition some finite exponential moment (used to check (\ref{3.2})) we show in Theorem \ref{theo5.4} that
\begin{equation}\label{0.15}
 h_* < \sqrt{2u_*} \,.
\end{equation}

\n
We also provide a sufficient condition for $h_* > 0$ in Theorem \ref{theo5.8}.  We show in Theorem \ref{theo5.8} that $h_* > 0$ when the offspring distribution has mean $m > 2$. Whereas Proposition \ref{prop4.2}  relies on  the existence of an infinite binary sub-tree of sites having uniformly bounded degree, Theorem \ref{theo5.8} follows a strategy in the spirit of Tassy \cite{Tass10} for random interlacements on Galton-Watson trees, but the situation is more complicated in the case of the Gaussian free field.  One can naturally wonder whether $h_* > 0$ holds generally when $m>1$, see also Remark \ref{rem5.9} .

\medskip
We now explain the organization of this article. In Section 1 we introduce further notation and recall various facts concerning the Gaussian free field and random interlacements. In Section 2 we consider the Gaussian free field $\wt{\varphi}$ on the cable system attached to $T$. We introduce condition (\ref{2.3}) (see also (\ref{0.7})), which enables us to prove in Proposition \ref{prop2.2} that $\{\wt{\varphi} > 0\}$ only has bounded components, and hence to apply the results of \cite{LupuA} and \cite{Szni}. We show (\ref{0.8}) in Corollary \ref{cor2.3}. In Section 3 we introduce the conditions (\ref{3.1}), (\ref{3.2}) (see (\ref{0.10}), (\ref{0.11})), and show in Theorem \ref{theo3.3} that together with the assumption $0 < u_* < \infty$, they imply that $h_* < \sqrt{2u_*}$. An important step established in Proposition \ref{prop3.2} shows the exponential decay of $\IP^G[x_0 \overset{\varphi \ge 0}{\longleftrightarrow} x]$ under (\ref{3.1}) (i.e.~(\ref{0.10})). Section 4 provides a sufficient condition for $h_* > 0$ in Proposition \ref{prop4.2} and the proof of (\ref{0.13}) follows from Corollary \ref{cor4.5}, as explained in Remark \ref{rem4.6}. In Section 5 applications to super-critical Galton-Watson trees conditioned on non-extinction are discussed. The fact that $h_*$ is deterministic (i.e.~almost surely constant conditioned on non-extinction) appears in Lemma \ref{lem5.1}. The key Proposition \ref{prop5.2} establishes that (\ref{3.1}) holds almost surely on non-extinction, and (\ref{0.14}) comes as a by-product, see also Theorem \ref{theo5.4}. Then, Proposition \ref{prop5.3} shows that under the finiteness of some exponential moment of the offspring distribution, condition (\ref{3.2}) holds almost surely on non-extinction. From that (\ref{0.15}) readily follows in Theorem \ref{theo5.4}. In Theorem \ref{theo5.8} we give sufficient conditions for $h_* > 0$. Finally, the Appendix contains the proof of the inequality $h_* \ge 0$ in the general set-up of transient weighted graphs, along similar arguments as in \cite{BricLeboMaes87}.

\section{Some preliminaries}
\setcounter{equation}{0}

In this section we introduce further notation and collect useful results concerning transient trees, random walks, level-set percolation of the Gaussian free field, and percolation of the vacant set of random interlacements.

\medskip
We consider a locally finite tree $T$ with root $x_0$ such that each edge has unit weight and the resulting network is transient. We write $x \sim y$ when $x$ and $y$ are neighbors in $T$, we let $d(\cdot,\cdot)$ stand for the geodesic distance on the tree and $|x| = d(x,x_0)$ stand for the height of $x$ in $T$. Given $U\subseteq T$, we let $\partial U = \{y \in T \backslash U$; $d(y,x) = 1$, for some $x \in U\}$ and $\partial_i U = \{x \in U$; $d(y,x) = 1$, for some $y \in T \backslash U\}$ respectively denote the outer and inner boundary of $U$. We let $|U|$ stand for the cardinality of $U$. For $x$ in $T$ we write $d(x,U) = \inf\{d(x,y)$; $y \in U\}$ for the distance of $x$ to $U$. Given $x,y$ in $T$, we let $[x,y]$ stand for the collection of sites on the geodesic path from $x$ to $y$. We also use the notation $(x,y]$, $[x,y)$, or $(x,y)$ when we exclude one or both endpoints. When $x \not= x_0$ we let $x^-$ stand for the last point before $x$ on the geodesic path from $x_0$ to $x$. Given $x \in T$, we denote by $d_x = |\{y \in T$; $y^- = x\}|$ the number of descendants of $x$, so that ${\rm deg}(x)$, the degree at $x$, equals $d_{x_0}$ when $x = x_0$ and $d_x + 1$ when $x \not= x_0$. We let $T_x$ stand for the sub-tree of descendents of $x$, i.e.~consisting of those $y$ in $T$ for which $x$ belongs to $[x_0,y]$. As far as dependence on the choice of the root is concerned, note that if a new root $x'_0$ is chosen, then $T_x$ remains unchanged as soon as $x \notin [x_0,x'_0]$. Finally, a cut-set $C$ separating $x_0$ from infinity (we will write cut-set for short) is a finite subset of $T \backslash \{x_0\}$, such that $x \not= y$ in $C$ implies that $x \notin T_y$ (and $y \notin T_x)$, and the connected component $U_C$ of $x_0$ after deletion of the edges $\{x^-,x\}$, $x \in C$, is finite. We write $B_C = U_C \cup C$.

\medskip
We now introduce some notation concerning simple random walk and potential theory on $T$. Given $U \subseteq T$, we write $T_U = \inf\{k \ge 0$; $X_k \notin U\}$ for the exit time of $U$, $H_U = \inf\{k \ge 0; X_k \in U\}$ for the entrance time in $U$, and $\wt{H}_U = \inf\{k \ge 1$;  $X_k \in U\}$ for the hitting time of $U$ of the canonical walk $(X_k)_{k \ge 0}$ on $T$.

\medskip
With similar notation as in (\ref{0.1}), the Green function killed outside $U$ is
\begin{equation}\label{1.1}
g_U (x,y) = \mbox{\f $\dis\frac{1}{{\rm deg}(y)}$} \;E_x \big[\dsl_{0 \le k < T_U} 1\{X_k = y\}\big], \; \mbox{for $x,y \in T$}.
\end{equation}

\medskip\n
It is symmetric and vanishes when $x$ or $y$ does not belong to $U$. When $U = T$, we recover the Green function $g(x,y)$ from (\ref{0.1}).

\bigskip
For $K$ finite subset of $T$, the equilibrium measure of $K$ is defined as
\begin{equation}\label{1.2}
e_K(x) = {\rm deg}(x) \;P_x[\wt{H}_K = \infty] \,1_K(x), \;\mbox{for $x \in T$}.
\end{equation}
It is concentrated on the inner boundary of $K$ and satisfies the identity
\begin{equation}\label{1.3}
P_x [H_K < \infty] = \dsl_y g(x,y) \, e_K(y), \; \mbox{for $x \in T$}.
\end{equation}
The total mass of $e_K$ is the capacity ${\rm cap}(K)$ of $K$.

\medskip
As mentioned in the Introduction, an important quantity for $x$ in $T$ is the positive (possibly infinite) quantity
\begin{equation}\label{1.4}
\mbox{$R^\infty_x =$ the effective resistance between $x$ and $\infty$ in $T_x$}
\end{equation}

\n
(in particular $R^\infty_x = \infty$ when $|T_x| < \infty$, and $R^\infty_x$ is the non-decreasing limit in $N$ of the effective resistance in $T_x$ between $x$ and $\{x' \in T_x$; $d(x,x') = N\}$, when $|T_x| = \infty$).

\medskip
As an aside, note that by the observation made above (\ref{1.1}) moving the root $x_0$ to a different location $x'_0$ will only change finitely many of the $R^\infty_x$, $x \in T$. We then define
\begin{equation}\label{1.5}
\alpha_x = \mbox{\f $\dis\frac{R^\infty_x}{1 + R^\infty_x}$} \in (0,1], \; \mbox{for $x \in T$},
\end{equation}
as well as for $0 < \alpha \le 1$ the operator
\begin{equation}\label{1.6}
Q^\alpha f(a) = E^Y [f(\alpha a + \sqrt{\alpha} \,Y)], \; \mbox{for $a \in \IR$},
\end{equation}

\n
where $Y$ stands for a standard normal variable, $E^Y$ for the corresponding expectation, and $f$ for a bounded measurable function. Note that for $\alpha = 1$, the above $Q^\alpha$ coincides with the Brownian transition kernel at time $1$.

\medskip
We now turn to the Gaussian free field $\varphi$ on $T$. For $U \subseteq T$ we denote by $\sigma_U$ the $\sigma$-algebra
\begin{equation}\label{1.7}
\sigma_U = \sigma(\varphi_x, x \in U).
\end{equation}

\n
From the Markov property of the Gaussian free field, one knows that for $x,y$ in $T$ with $y^- = x$,
\begin{align}
&(\varphi_{y'} - P_{y'} [H_x < \infty] \,\varphi_x)_{y' \in T_y} \;\mbox{is a centered Gaussian field with} \label{1.8}
\\
&\mbox{covariance $g_{U = T_y}(\cdot,\cdot)$ independent of $\sigma_{T \backslash T_y}$}. \nonumber
\end{align}

\n
The next lemma relates the objects we have now introduced, and will be recurrently used in this work ((\ref{1.15}) will be used in the proof of Proposition \ref{prop2.2} in Section 2).

\begin{lemma}\label{lem1.1}  For $x$ in $T$, one has
\begin{align}
&\hspace{-3.7cm}g(x,x) \le R^\infty_x, \;\mbox{with equality when $x = x_0$}, \label{1.9}
\\[2ex]
\label{1.10}
\left\{ \begin{array}{ll}
{\rm i)} &g(x,x) \ge 1/{\rm deg}(x), 
\\[1ex]
{\rm ii)} &R^\infty_x \ge 1 / d_x.
\end{array}\right.
\end{align}

\n
For $x,y$ in $T$ with $y^- = x$,
\begin{align}
P_y[H_x < \infty] &= \alpha_y, \; \; P_y [H_x = \infty] = (1 + R^\infty_y)^{-1}, \label{1.11}
\\[1ex]
g_{U = T_y} (y,y) &= \alpha_y, \label{1.12}
\end{align}
and for any bounded measurable function $f$ on $\IR$ one has
\begin{equation}\label{1.13}
\IE^G [f(\varphi_y) \, | \, \sigma_{T \backslash T_y}]= Q^{\alpha_y} f(\varphi_x).
\end{equation}
When $C$ is a cut-set, one has the identities
\begin{align}
& e_C(x) = \dis\frac{1}{R^\infty_x} \; \;\mbox{when $x \in C$, and}\label{1.14}
\\
&1 = \dsl_{x \in C} g(x_0,x) \;  \dis\frac{1}{R^\infty_x}.\label{1.15}
\end{align}
\end{lemma}

\begin{proof}
The claims (\ref{1.9}) and (\ref{1.10}) follow from the fact that $g(x,x)$ is the effective resistance between $x$ and infinity in $T$, whereas $R^\infty_x$ is the effective resistance between $x$ and infinity in $T_x$. As for (\ref{1.11}), set $T'_y = \{x\} \cup T_y$, then the effective conductance between $x$ and infinity in the sub-tree $T'_y$ coincides with the escape probability $P_y[H_x = \infty]$, see also \cite{LyonPere16}, above Proposition 17.26, so that $P_y[H_x = \infty] = (1 + R^\infty_y)^{-1}$ and $P_y[H_x < \infty] = \alpha_y$. Concerning (\ref{1.12}), note that $g_{U = T_y} (y,y)$ coincides with the effective resistance between $y$ and $\{x\} \cup \{\infty\}$ in $T'_y$, so that $g_{U = T_y}(y,y) = (1+ \frac{1}{R^\infty_y})^{-1} = \alpha_y$, whence (\ref{1.12}).

\medskip
We now turn to (\ref{1.13}). By (\ref{1.8}) we know that $\varphi_y - P_y [H_x < \infty] \,\varphi_x$ is a Gaussian variable with variance $g_{U = T_y}(y,y)$. By (\ref{1.11}), (\ref{1.12}), and the formula (\ref{1.6}) defining $Q^\alpha$ the claim (\ref{1.13}) readily follows. Concerning (\ref{1.14}), we recall the notation $B_C$ for $C$ a cut-set, see above (\ref{1.1}). One has the equality
\begin{equation} \label{1.16}
e_{B_C} = e_C,
\end{equation}
so that
\begin{equation} \label{1.17}
e_C(x) \stackrel{(\ref{1.2})}{=} \dsl_{y^- = x} P_y[H_x = \infty] \stackrel{(\ref{1.11})}{=} \dsl_{y^- = x} (1 + R^\infty_y)^{-1} = \dis\frac{1}{R^\infty_x}, \;\mbox{for $x \in C$}.
\end{equation}

\medskip\n
Moreover, (\ref{1.15}) is now the direct application of (\ref{1.3}) (with the choice $x = x_0$, $K = B_C$) together with (\ref{1.14}). This concludes the proof of Lemma \ref{lem1.1}. 
\end{proof}

\medskip
\begin{remark}\label{rem1.2} \rm
Incidentally, when $y_n$, $0 \le n < N$, with $N < \infty$ or $N = \infty$, is a finite or semi-infinite geodesic path in $T$ moving away from the root $x_0$, and $R^\infty_{y_n} = \infty$ for each $1 \le n < N$, it follows from (\ref{1.13}) and from the observation made below (\ref{1.6}) that $(\varphi_{y_n})_{0 \le n < N}$ under $\IP^G$ is distributed as a Brownian motion with the initial law $N(0,g(y_0,y_0))$, sampled at the integer times $0 \le n < N$. \hfill $\square$
\end{remark}

We now continue with level-set percolation of the Gaussian free field. Given $x$ in $T$, and $h$ in $\IR$, we denote by $\{x \overset{\varphi \ge h}{\longleftrightarrow} \infty\}$ the event that the connected component of $\{\varphi \ge h\}$ containing $x$ is infinite. If $\IP^G[ x \overset{\varphi \ge h}{\longleftrightarrow} \infty] > 0$ and $y$ is neighbor of $x$, it is straightforward with (\ref{1.8}) (where $x$ plays the role of $x_0$) to infer that $\IP^G[y  \overset{\varphi \ge h}{\longleftrightarrow} \infty] > 0$ (one can also use the FKG-Inequality, see the Appendix of \cite{Giac03}). In other words, if $\IP^G[ x \overset{\varphi \ge h}{\longleftrightarrow} \infty]$ vanishes for some $x$ in $T$, it vanishes for all $x$ in $T$, and so we can express the critical value $h_*$ defined in (\ref{0.3}) as
\begin{equation}\label{1.18}
h_* = \inf\big\{h \in \IR; \;\IP^G \big[x_0 \overset{\varphi \ge h}{\longleftrightarrow} \infty] = 0\big\} \; \mbox{(with $x_0$ the root)}.
\end{equation}

\n
By an argument of \cite{BricLeboMaes87} one knows (actually, in the general set-up of transient weighted graphs, see Proposition \ref{propA2} of the Appendix) that
\begin{equation}\label{1.19}
0 \le h_* \le \infty.
\end{equation}
\n
Incidentally, in the case of $\IZ^d$, $d \ge 3$, one knows that $h_* < \infty$ for all $d \ge 3$, see \cite{BricLeboMaes87}, \cite{RodrSzni13}, but $h_* >0$ has only been proved  when $d$ is large enough, see \cite{RodrSzni13}, \cite{DrewRodr15}.

\bigskip
To further characterize $h_*$, we will now construct for each $h \in \IR$ and $x$ in $T$ a \linebreak $[0,1]$-valued function $q_{x,h}(\cdot)$, which is a ``good version'' of the conditional expectation \linebreak $\IP^G[x \overset{\!\!\!\!\!\!\!T_x,\varphi \ge h}{\longleftrightarrow  \hspace{-3ex}\mbox{\scriptsize /}\quad\infty} \,| \varphi_x= \cdot]$, where $\{x \overset{\!\!\!\!\!\!\!T_x,\varphi  \ge h}{\longleftrightarrow  \hspace{-3ex}\mbox{\scriptsize /}\quad\infty} \}$ refers to the event that the connected component of $x$ in $T_x \cap \{ \varphi \ge h\}$ is finite.

\medskip
For the definition we will now give, it is convenient to broaden the set-up, so that $T$ is a tree with root $x_0$, which is possibly recurrent or even finite. If $T$ is recurrent then we set $R^\infty_x = \infty$ and $\alpha_x = 1$, for all $x \in T$.

\medskip
For each $n \ge 0$, we write
\begin{equation}\label{1.20}
T_n = \{x \in T; \,|x| = n\}, \; B_n = \{x \in T; \,|x| \le n\}
\end{equation}
(so $T_n$ is possibly empty, when $T$ is finite).

\medskip
Then, for each $n \ge 0$, $h \in \IR$, $x \in B_n$, we define the functions $q^n_{x,h}(\cdot)$ by recursion towards the root $x_0$, starting from the boundary $T_n$, via
\begin{equation}\label{1.21}
\left\{ \begin{array}{l}
q^n_{x,h}(a) = 1_{(-\infty,h)}(a), \; \mbox{for $x \in T_n$}
\\[1ex]
q^n_{x,h}(a) = 1_{(-\infty,h)}(a) + 1_{[h,\infty)}(a) \prod\limits_{y^- = x} Q^{\alpha_y} (q^n_{y,h})(a), \; \mbox{for $|x| < n$},
\end{array}\right.
\end{equation}

\medskip\n
and an empty product (when $x$ has no descendent) is understood as equal to $1$.

\medskip
Note that when $T$ is finite and $n \ge 1$ such that $T_n = \phi$, then
\begin{equation}\label{1.22}
q^n_{x,h}(\cdot) = 1, \; \mbox{for all $x \in T$}.
\end{equation}

\begin{lemma}\label{lem1.2} ($T$ a possibly finite or recurrent tree with root $x_0$)

\medskip\n
The functions $q^n_{x,h}(a)$, for $n \ge |x|$, are non-increasing in $a$, $[0,1]$-valued, equal to $1$ on $(-\infty,h)$, with only possible discontinuity at $h$. For a fixed $x \in T$, they increase with $n \ge |x|$, and converge to a function $q_{x,h}(a)$ with similar properties, and such that
\begin{equation}\label{1.23}
q_{x,h} = 1_{(-\infty,h)} + 1_{[h,\infty)} \pr\limits_{y^- = x} Q^{\alpha_y} (q_{y,h}), \; \mbox{for all $x \in T$}
\end{equation}

\n
(and an empty product is understood as equal to $1$).

\medskip\n
When $T$ is finite, then
\begin{equation}\label{1.24}
q_{x,h} = 1, \; \mbox{for all $h \in \IR$, $x \in T$}.
\end{equation}

\medskip\n
When $T$ is transient, then for any $h \in \IR$, $x \in T$, one has
\begin{align}
q^n_{x,h}(\varphi_x) & \stackrel{\IP^G{\rm -a.s.}}{=} \IP^G\big[x \overset{\mbox{\scriptsize $T_x, \varphi \ge h$} \atop \;}{\longleftrightarrow} \hspace{-4.5ex}\mbox{\scriptsize /}\quad \;  T_n \,| \, \varphi_x\big], \; \mbox{for $n \ge |x|$}, \label{1.25}
\\[2ex]
q_{x,h}(\varphi_x) & \stackrel{\IP^G{\rm -a.s.}}{=} \IP^G\big[x \overset{\mbox{\scriptsize $T_x, \varphi \ge h$} \atop \;}{\longleftrightarrow} \hspace{-4.5ex}\mbox{\scriptsize /}\quad \; \infty \,| \, \varphi_x\big]. \label{1.26}
\end{align}
In addition, one has the dichotomy
\begin{equation}\label{1.27}
\left\{ \begin{array}{rl}
{\rm i)} & q_{x_0,h} = 1 \; \mbox{for $h > h_*$},
\\[1ex]
{\rm ii)} &q_{x_0,h} \; \mbox{is not identically $1$ for $h < h_*$}.
\end{array}\right.
\end{equation}
\end{lemma}

\begin{proof}
From the definition of $Q^\alpha$ in (\ref{1.6}) and the recursion from the boundary (\ref{1.21}), it is immediate that $q^n_{x,h}(\cdot)$ are non-increasing, $[0,1]$-valued functions, equal to $1$ on $(-\infty,h)$, with only possible discontinuity at $h$. When $x \in T_n$, $q^n_{x,h}(\cdot) \le q^{n+1}_{x,h}(\cdot)$ by (\ref{1.21}) and this gets propagated inside $B_n$ by the recursion (\ref{1.21}), so that $q^n_{x,h}(\cdot) \le q^{n+1}_{x,h}(\cdot)$ for $x \in B_n$ (when $T_n = \phi$, actually (\ref{1.22}) holds). Setting $q_{x,h}(a) = \lim_n \uparrow q^n_{x,h}(a)$, we obtain (\ref{1.23}) from (\ref{1.21}) by monotone convergence. It also follows that $q_{x,h}(\cdot)$ is non-increasing $[0,1]$-valued, with value $1$ on $(-\infty,h)$ and only possible discontinuity at $h$ (due to (\ref{1.23}) and (\ref{1.6})). The claim (\ref{1.24}) for finite $T$ is immediate from (\ref{1.22}).

\medskip
Let us now assume that $T$ is transient and prove (\ref{1.25}). We fix $n$ and use induction on $n - |x|$. When $x \in T_n$, then (\ref{1.25}) is immediate from the first line of (\ref{1.21}). When $|x| < n$, then one has the $\IP^G$-a.s. equality
\begin{equation}\label{1.28}
\IP^G\big[x \overset{\mbox{\scriptsize $T_x, \varphi \ge h$} \atop \;}{\longleftrightarrow} \hspace{-4.5ex}\mbox{\scriptsize /}\quad \;  T_n \,| \, \varphi_x\big] = 1_{(-\infty,h)}(\varphi_x) + 1_{[h,\infty)}(\varphi_x) \,\IP^G \Big[\bigcap\limits_{y^- = x} \big\{ y \overset{\mbox{\scriptsize $T_y, \varphi \ge h$} \atop \;}{\longleftrightarrow} \hspace{-4.5ex}\mbox{\scriptsize /}\quad \;  T_n\big\}\,| \,\varphi_x\Big].
\end{equation}

\n
If one first conditions on $\varphi_x$ and $\varphi_y$, for $y^- = x$ it follows from the Markov property (\ref{1.8}) and induction that we have $\IP^G$-a.s.,
\begin{equation*}
\IP^G\Big[\bigcap_{y^- = x} \big\{y \overset{\mbox{\scriptsize $T_y, \varphi \ge h$} \atop \;}{\longleftrightarrow} \hspace{-4.5ex}\mbox{\scriptsize /}\quad \;  T_n\big\} \,| \,\varphi_x\Big]  = \IE^G \big[\pr_{y^- = x} q^n_{y,h}(\varphi_y) \, | \, \varphi_x\big]
\stackrel{(\ref{1.8}),(\ref{1.13})}{=} \pr_{y^- = x} Q^{\alpha_y} (q^n_{y,h})(\varphi_x).
\end{equation*}

\n
Inserting this identity in the last expression of (\ref{1.28}), and using (\ref{1.21}), we see that (\ref{1.25}) holds for $x$ as well. This completes the proof of (\ref{1.25}) by induction. The claim (\ref{1.26}) readily follows by monotone convergence.

\medskip
Finally, let us prove (\ref{1.27}). We know that for $h > h_*$, $\IP^G [x_0 \overset{\varphi \ge h}{\longleftrightarrow} \hspace{-3ex}\mbox{\scriptsize /}\quad \infty] = 1$, cf.~(\ref{1.18}). By (\ref{1.26}) we see that $q_{x_0,h}(\cdot) = 1$ almost everywhere and hence everywhere due to the fact that $q_{x_0,h}(\cdot)$ is non-increasing. This proves (\ref{1.27}) i). On the other hand, for $h < h_*$, $\IP^G[x_0 \overset{\varphi \ge h}{\longleftrightarrow} \hspace{-3ex}\mbox{\scriptsize /}\quad \infty] < 1$, and by (\ref{1.26}), $q_{x_0,h}(\cdot)$ is not identically $1$, whence (\ref{1.27}) ii). This completes the proof of Lemma \ref{lem1.2}.
\end{proof}

\begin{remark}\label{rem1.3} \rm  Note that the recursion (\ref{1.21}) used in the construction of the functions $q^n_{x,h}$ only involves the coefficients $\alpha_y$, for $y \in T_x$. If we write $q^T_{x,h}$ for $q_{x,h}$, with $x \in T$ and $h \in \IR$, to underline the dependence in $T$, it is straightforward to infer from (\ref{1.21}) and the above observation that for all $x \in T$ and $h \in \IR$
\begin{equation}\label{1.29}
q^T_{x,h} = q^{T_x}_{x,h}.
\end{equation}

\n
This identity combined with (\ref{1.27}) will be useful when proving that $h_*$ is almost surely constant for a super-critical Galton-Watson tree conditioned on non-extinction. \hfill $\square$
\end{remark}

We now return to the case where $T$ is a transient tree with root $x_0$, and consider the sub-tree (with same root $x_0$) of vertices with an infinite line of descent
\begin{equation}\label{1.30}
T^\infty = \{x \in T; \,|T_x| = \infty\}.
\end{equation}

\medskip\n
Then, the connected components of $T \backslash T^\infty$ consist of finite sub-trees, and $T^\infty$ is a transient tree with Green's function equal to the restriction of $g(\cdot,\cdot)$ to $T^\infty$, see for instance Proposition 1.11 of \cite{Szni12e}. Thus, the law of $(\varphi_x)_{x \in T^\infty}$ under $\IP^G$ equals the law of the Gaussian free field on $T^\infty$. Note also that for any $h \in \IR$ one has $\{x_0 \stackrel{\varphi \ge h}{\longleftrightarrow} \infty\} = \{x_0 \stackrel{T^\infty,\varphi \ge h}{\longleftrightarrow} \infty\}$ (where this last notation refers to the event that the connected component of $x_0$ in $T^\infty \cap \{\varphi \ge h\}$ is infinite), so that with hopefully obvious notation
\begin{equation}\label{1.31}
h_* (T) = h_*(T^\infty),
\end{equation}

\n
i.e.~the critical values for level-set percolation of the Gaussian free field on $T$ and on $T^\infty$ coincide.

\medskip
We now briefly turn to the topic of random interlacements on $T$ and recall some facts concerning the percolation of the vacant set of random interlacements. We refer to the monographs \cite{CernTeix12} and \cite{DrewRathSapo14c} for further material and references. The vacant set of random interlacements at level $u \ge 0$ on $T$ is a random subset $\cV^u$ of $T$, governed by a probability $\IP^I$, with law characterized by
\begin{equation}\label{1.32}
\IP^I[\cV^u \supseteq K] = \exp\{-u \, {\rm cap}(K)\}, \;\mbox{for any finite $K \subseteq T$}
\end{equation}

\n
(with ${\rm cap}(K)$ the capacity of $K$, see below (\ref{1.3})).

\medskip
As $u$ increases, $\cV^u$ becomes thinner, and to classify the percolative properties of $\cV^u$, one defines $u_*$ as in (\ref{0.5}). Actually, one has regardless of the choice of the base point $x_0$, see Corollary 3.2 of \cite{Teix09b},
\begin{equation}\label{1.33}
u_* = \inf\{ u \ge 0; \, \IP^I [x_0 \stackrel{\cV^u}{\longleftrightarrow} \infty] = 0\} \in [0,\infty].
\end{equation}

\n
One also knows by Theorem 5.1 of \cite{Teix09b} (the bounded degree assumption stated there can be removed) that
\begin{equation}\label{1.34}
\begin{array}{l}
\mbox{the connected component $\cC^{\cV^u}(x_0)$ of $x_0$ in $\cV^u$ has the same law as the} \\
\mbox{open cluster of $x_0$ in an independent site Bernoulli percolation on $T$,}\\
\mbox{for which each site $x \in T$ is open with probability $p_{x,u}$},
\end{array}
\end{equation}
where
\begin{equation}\label{1.35}
\left\{ \begin{split}
p_{x_0,u} = &\;e^{-u \,{\rm cap}(\{x_0\})}, \;\mbox{and for $x \not= x_0$}
\\[1ex]
p_{x,u} = &\; e^{-u \, {\rm deg}(x) \;P_x[d(X_n,x_0) > d(x,x_0), \; {\rm for \; all} \;n > 0] \, \cdot \,P_x[d(X_n,x_0) \ge d(x,x_0), \;{\rm for \; all} \;n \ge 0]}
\end{split}\right.
\end{equation}

\n
Taking into account that ${\rm cap}(\{x_0\}) = g(x_0,x_0)^{-1}$ as well as (\ref{1.9}), (\ref{1.11}), we see that for $u \ge 0$,
\begin{equation}\label{1.36}
p_{x_0,u} = e^{-\frac{u}{R^\infty_{x_0}}} \;  \mbox{and for $x \not= x_0$, $p_{x,u} = e^{-\frac{u}{R^\infty_x (1 + R^\infty_x)}}$} \,.
\end{equation}

\begin{remark}\label{rem1.4} \rm If $T^\infty$ stands for the sub-tree of vertices with an infinite line of descent, see (\ref{1.30}), then with hopefully obvious notation
\begin{equation}\label{1.37}
R_x^\infty(T) = R^\infty_x (T^\infty) \; \mbox{for all $x \in T^\infty$}
\end{equation}

\n
(all components of $T \backslash T^\infty$ are finite, and (\ref{1.37}) can be seen by replacing in the approximation of $R^\infty_x(T)$ below (\ref{1.4}) the set $\{x' \in T_x$; $d(x,x') = N\}$ by the set $\{x' \in T^\infty_x$; $d(x,x') = N\}$). In particular, in view of (\ref{1.36}), we find that with similar notation as in (\ref{1.37}) one has
\begin{equation}\label{1.38}
p_{x,u}(T) = p_{x,u}(T^\infty)\; \mbox{for all $x \in T^\infty$},
\end{equation}
and in view of (\ref{1.33}), (\ref{1.34}),
\begin{equation}\label{1.39}
u_*(T) = u_* (T^\infty),
\end{equation}

\medskip\n
i.e.~the critical values for the percolation of the vacant set of random interlacements on $T$ and on $T^\infty$ coincide. \hfill $\square$
\end{remark}

Let us close this section by mentioning that percolation of $\cV^u$ can be re-expressed in terms of the transience of $T$ endowed with certain weights. More precisely, if one introduces on the edges $e = \{x^-,x\}$, for $x \in T \backslash \{x_0\}$ the weights
\begin{equation}\label{1.40}
c_u(e) = e^{-u \sum\limits_{y \in (x_0,x]}  \, \frac{1}{R_y^\infty (1 + R^\infty_y)}} \; \big(1 - e^{-u \frac{1}{R^\infty_x (1 + R^\infty_x)}  } \big)^{-1},
\end{equation}

\n
then, when $u > 0$ and $R^\infty_x < \infty$ for each $x \in T$, one knows from Theorem 2.1 of \cite{Lyon92}, see also Corollary 5.25 of \cite{LyonPere16}, and (\ref{1.34}), (\ref{1.36}) that
\begin{equation}\label{1.41}
\mbox{$\IP^I[x_0 \stackrel{\cV^u}{\longleftrightarrow} \infty] > 0$ if and only if $T$ endowed with the weights (\ref{1.40}) is transient}.
\end{equation}

\section{Some consequences of the cable methodology}
\setcounter{equation}{0}

In this section we consider the Gaussian free field $\wt{\varphi}$ on the cable system $\wt{T}$ attached to $T$. We use it to infer as an application of the results of \cite{LupuA} and \cite{Szni} the inequality $h_* \le \sqrt{2u_*}$, as well as a coupling, which relates the level sets of the Gaussian free field and the vacant set of random interlacements on $T$, see Corollary \ref{cor2.3}. This coupling will be the main tool in the next section to derive (under suitable assumptions) the strict inequality $h_* < \sqrt{2 u}_*$. In Remark \ref{rem2.4} 2) we also provide an example where $h_*$ and $u_*$ vanish.

\medskip
As in the previous section, $T$ is a transient tree with base point $x_0$, such that each edge has unit conductance. The cable tree $\wt{T}$ is obtained by attaching to each edge $e = \{x,y\}$ of the tree a compact interval with length $\frac{1}{2}$ and endpoints respectively identified to $x$ and $y$. One defines on $\wt{T}$ a continuous diffusion behaving as a standard Brownian motion in the interior of each such segment. It has a continuous symmetric Green function $\wt{g}(z,z')$, $z,z' \in \wt{T}$ with respect to the Lebesgue measure on $\wt{T}$, which extends the Green function $g(\cdot,\cdot)$ of the discrete time walk on $T$, see (\ref{0.1}). We refer to Section 2 of \cite{LupuA}, Section 2 of \cite{Folz14}, and Section 3 of \cite{Zhai} for more details.

\medskip
We now turn to the Gaussian free field on the cable tree $\wt{T}$. On the canonical space $\wt{\Omega}$ of continuous real-valued functions on $\wt{T}$ endowed with the $\sigma$-algebra generated by the canonical coordinates $\wt{\varphi}_z$ (we also sometimes write $\wt{\varphi}(z)$), $z \in \wt{T}$, we denote by $\wt{\IP}^G$ the probability, with corresponding expectation $\wt{\IE}^G$, such that
\begin{equation}\label{2.1}
\begin{array}{l}
\mbox{under $\wt{\IP}^G$, $(\wt{\varphi}_z)_{z \in \wt{T}}$ is a centered Gaussian field with}
\\
\mbox{covariance $\wt{\IE}^G[\wt{\varphi}_z \wt{\varphi}_{z'}] = \wt{g}(z,z')$}.
\end{array}
\end{equation}

\n
In particular, looking at the restriction of $\wt{\varphi}$ to $T$, we see that
\begin{equation}\label{2.2}
\mbox{the law of $(\wt{\varphi}_x)_{x \in T}$ under $\wt{\IP}^G$ is equal to $\IP^G$}.
\end{equation}

\n
An important issue in this context is to establish that $\wt{\IP}^G$-a.s., $\{\wt{\varphi} > 0\}$ only has bounded components in $\wt{T}$. As shown in \cite{LupuA}, see also (1.33) of \cite{Szni}, when this condition holds, then for $u > 0$ one can couple $\{\varphi > \sqrt{2u}\}$ and $\cV^u$ so that $\{\varphi > \sqrt{2u}\} \subseteq \cV^u$, a.s.. It then follows that $h_* \le \sqrt{2u_*}$. 

\medskip
We will introduce a condition, see (\ref{2.3}), which implies the above condition, but also enables us to apply the results of Section 2  of \cite{Szni} (see Corollary 2.5 and Remark 2.6 therein), and construct a strengthened coupling between $\{\varphi > \sqrt{2u}\}$ and $\cV^u$, see (\ref{2.20}) below.

\medskip
We thus introduce the condition, see (\ref{1.4}) for notation,
\begin{equation}\label{2.3}
\mbox{for some $A > 0$, the set $\{x \in T$; $R^\infty_x > A\}$ only has bounded components}.
\end{equation}

\begin{remark}\label{rem2.1} \rm ~

\medskip\n
1) The condition (\ref{2.3}) as we now explain is equivalent to the existence of a sequence of cut-sets $C_n$, $n \ge 1$, and $A > 0$, such that (see above (\ref{1.1}) for notation)
\begin{equation}\label{2.4}
\left\{ \begin{array}{ll}
{\rm i)} & B_{C_n} \subseteq U_{C_{n+1}} \;\mbox{for each $n \ge 1$}
\\[2ex]
{\rm ii)} & \sup\limits_{n \ge 1} \; \sup\limits_{x \in C_n} \; R^\infty_x \le A.
\end{array}\right.
\end{equation}

\medskip\n
Indeed, (\ref{2.4}) readily implies (\ref{2.3}). Conversely, when (\ref{2.3}) holds, one defines $U_1$ consisting of $x_0$ and the points linked to $x_0$ by a path where $R^\infty_x > A$ prior to reaching $x_0$, and sets $C_1 = \partial U_1$. By induction one then defines $U_{n+1}$ as the union of $U_n,C_n$ and the collection of points linked to $C_n$ by a path where $R^\infty_x > A$ prior to reaching $C_n$, and sets $C_{n+1} = \partial U_{n+1}$. Then $B_{C_n} = U_n \cup C_n$, for each $n \ge 1$, and (\ref{2.4}) holds.

\medskip
Let us also mention that when $C_n$ is a sequence of cut-sets as in (\ref{2.4}), then by i)
\begin{equation}\label{2.5}
\mbox{$B_{C_n} \subseteq B_{C_{n+1}}$ and $d(x_0, C_n) \ge n$, for all $n \ge 1$}.
\end{equation}

\medskip\n
2) As a result of the observation above (\ref{1.5}), condition (\ref{2.3}) does not depend on the choice of the base point $x_0$ in $T$. \hfill $\square$
\end{remark}

The main result established in this section comes in the next proposition. Its consequences appear in Corollary \ref{cor2.3}.

\begin{proposition}\label{prop2.2}
Assume that (\ref{2.3}) holds, then
\begin{equation}\label{2.6}
\mbox{$\wt{\IP}^G$-a.s., $\{\wt{\varphi}> 0\}$ only has bounded components in $\wt{T}$}.
\end{equation}
\end{proposition}

\begin{proof}
For $x,y$ in $T$, we write $[\wt{x,y}]$ for the geodesic segment in $\wt{T}$ between $x$ and $y$. One has the following identities, which are consequences of the strong Markov property of $(\wt{\varphi}_z)_{z \in \wt{T}}$, see Lemma 3.1 and Proposition 5.2 of \cite{LupuA}: for $x \in T$,
\begin{align}
&\mbox{$\wt{\IP}^G\big[\wt{\varphi}$ does not vanish on $[\wt{x_0,x}]\big] = \mbox{\f $\dis\frac{2}{\pi}$} \;\arcsin \Big(\mbox{\f $\dis\frac{g(x_0,x)}{\sqrt{g(x_0,x_0)  g(x,x)}}$}\Big)$}, \label{2.7}
\\[2ex]
&\mbox{$\wt{E}^G\big[\wt{\varphi}_{x_0}\, \wt{\varphi}_x, \wt{\varphi}$ does not vanish on $[\wt{x_0,x}]\big] = g(x_0,x)$} \label{2.8}
\end{align}
(and the notation in (\ref{2.8}) refers to the product with the indicator function of the event following the comma).

\medskip\n
We consider $A > 0$ and a sequence of cut-sets $C_n$, $n \ge 1$, as in (\ref{2.4}). We will first work under the additional assumption that
\begin{equation}\label{2.9}
\mbox{$d_x = 1$, for all $x \in C_n$ and $n \ge 1$}
\end{equation}

\medskip\n
(where $d_x$ stands for the number of descendents of $x$ in $T$, see the beginning of Section 1). We will then treat the general case. We thus assume (\ref{2.9}) and define for each $n \ge 1$,
\begin{equation}\label{2.10}
\wt{\cZ}_n = \{ x \in C_n; \, \wt{\varphi} > 0 \; \mbox{on} \; [\wt{x_0,x}]\}.
\end{equation}

\n
In what follows, constants will possibly depend on $d_{x_0}$, $R^\infty_{x_0}$, $A$, and will change from line to line. Additional dependence will appear in the notation. By (\ref{1.9}), (\ref{1.10}) and (\ref{2.4}) ii) and (\ref{2.9}), we see that
\begin{equation}\label{2.11}
c \le g(x,x) \le c' \;\mbox{for} \; x = x_0 \;\mbox{and} \; x \in \textstyle{\bigcup\limits_{n \ge 1}} C_n .
\end{equation}
Thus, by (\ref{2.7}) and (\ref{2.8}), we see that for $n \ge 1$,
\begin{equation}\label{2.12}
\begin{split}
\wt{\IE}^G \Big[\dsl_{x \in C_n} (1 + \wt{\varphi}_{x_0} \,\wt{\varphi}_x) \;1\{\wt{\varphi} > 0 \; \mbox{on} \; [\wt{x_0,x}]\}\Big] & \le c \dsl_{x \in C_n} g(x_0,x)
\\[1ex]
&\!\!\!\!\stackrel{(\ref{2.4})\, {\rm ii)}}{\le} c' \dsl_{x \in C_n} g(x_0, x) \;\mbox{\f $\dis\frac{1}{R^\infty_x}$} \stackrel{(\ref{1.15})}{=} c'.
\end{split}
\end{equation}

\n
As a result, it follows from Fatou's lemma that
\begin{equation}\label{2.13}
\wt{\IE}^G\big[\liminf_n \dsl_{x \in C_n} (1 + \wt{\varphi}_{x_0} \,\wt{\varphi}_x) \,1 \{\wt{\varphi} > 0 \;\mbox{on} \; [\wt{x_0,x}]\}\big] \le c'.
\end{equation}

\n
This bound implies that the event 
\begin{equation*}
\textstyle{\bigcup\limits_{L \ge 1}}\, \limsup\limits_n \Big\{ | \wt{\cZ}_n| \le L, \; | \wt{\varphi}_{x_0} | \dsl_{x \in \wt{\cZ}_n} \wt{\varphi}_x \le L\Big\}
\end{equation*}
has full $\wt{\IP}^G$-probability. Since $|\wt{\varphi}_{x_0}| > 0$, $\wt{\IP}^G$-a.s., we find that
\begin{equation}\label{2.14}
\wt{\IP}^G \Big[ \textstyle{\bigcup\limits_{M \ge 1}} \; \limsup\limits_n \Big\{| \wt{\cZ}_n | \le M \;\mbox{and} \; \dsl_{x \in \wt{\cZ}_n} \varphi_x \le M\Big\}\Big] = 1.
\end{equation}

\n
Note that for any $0 < \alpha \le 1$, we have for any $0 \le a \le M$ (see (\ref{1.6}) for notation)
\begin{equation}\label{2.15}
\begin{split}
Q^\alpha 1_{(-\infty,0)} (a) & = P^Y[\alpha a + \sqrt{\alpha} \;Y < 0] = P^Y[Y < - \sqrt{\alpha}\,a]
\\
&\!\!\!\!\!\!\!\stackrel{a \le M, \alpha \le 1}{\ge} P^Y [Y < - M] = c(M).
\end{split}
\end{equation}

\n
We now introduce for $n, M \ge 1$ the event $A_{M,n} = \{\sum_{x \in \wt{\cZ}_n} (1 + \wt{\varphi}_x) \le M\}$ as well as the $\sigma$-algebra
\begin{equation}\label{2.16}
\wt{\cF}_n = \sigma \Big(\wt{\varphi}_z, z \in \textstyle{\bigcup\limits_{x \in C_n}} [\wt{x_0,x}]\Big).
\end{equation}

\n
It now follows from the Markov property of $\wt{\varphi}$, see (1.8) of \cite{Szni}, that on $A_{M,n}$
\begin{equation}\label{2.17}
\begin{split}
\wt{\IP}^G [ | \wt{\cZ}_{n+1}| = 0 \,| \, \wt{\cF}_n] & \ge \wt{\IP}^G [\wt{\varphi}_y < 0, \;\mbox{for all $y \in T$ with $y^- \in \wt{\cZ}_n\,| \, \wt{\cF}_n]$}
\\
& = \pr\limits_{x \in \wt{\cZ}_n} \; \pr\limits_{y^- = x} Q^{\alpha_y} (1_{(-\infty,0)})(\wt{\varphi}_x) \stackrel{(\ref{2.9}),(\ref{2.15})}{\ge} c(M)^M (\mbox{on} \; A_{M,n}).
\end{split}
\end{equation}

\n
By Borel-Cantelli's lemma it then follows that
\begin{equation*}
\mbox{$\wt{\IP}^G$-a.s., on $\limsup\limits_n A_{M,n}$, $|\wt{\cZ}_k| = 0$ for large $k$}.
\end{equation*}

\n
Since $\wt{\IP}^G [\bigcup_{M \ge 1} \limsup\limits_n A_{M,n}] = 1$ by (\ref{2.14}), we find that
\begin{equation*}
\mbox{$\wt{\IP}^G$-a.s., $|\wt{\cZ}_n| = 0$, for large $n$}.
\end{equation*}

\n
We have thus shown that under (\ref{2.9})
\begin{equation}\label{2.18}
\mbox{$\wt{\IP}^G$-a.s., the connected component of $x_0$ in $\{\wt{\varphi} > 0\}$ is bounded}.
\end{equation}

\medskip\n
We will now remove assumption (\ref{2.9}). In essence, we use a scaling argument. We denote by $T^*$ the tree with vertex set consisting of $T$ and the mid-points of the intervals in $\wt{T}$ linking neighboring vertices in $T$, and edges of unit weight linking each mid-point to the two end-points of the interval where it lies. If $\wt{T}^*$ denotes the corresponding cable graph, there is a natural bijection $s$ from $\wt{T}^*$ onto $\wt{T}$, which, in essence, ``scales by $\frac{1}{2}$'' each interval linking neighbors in $T^*$. The effective resistance between two points in $\wt{T}^*$ is then twice the effective resistance between their images in $\wt{T}$. Then, looking at Green functions, $\wt{g}^* (z^*_1,z^*_2) = 2 \wt{g} (s (z_1^*)$, $s(z_2^*))$, for $z^*_1,z_2^*$ in $\wt{T}^*$. It now follows that $(\sqrt{2} \,\wt{\varphi}_{s(z^*)})_{z^* \in \wt{T}^*}$ under $\wt{\IP}^G$ has the same law as the Gaussian free field on $\wt{T}^*$.

\medskip
Now consider a sequence $C_n \subseteq T$, $n \ge 1$, as in (\ref{2.4}), and denote by $C^*_n \subseteq T^*$, the collection of mid-points $x_*$ of the intervals attached to $\{x,x^-\}$, with $x \in C_n$, for $n \ge 1$. Note that $R^\infty_{x^*}(T^*) = 1 + 2 R^\infty_x(T)$, for such $x$ and $x_*$, and $C^*_n$, $n \ge 1$ is a sequence of cut-sets of $T^*$ satisfying (\ref{2.4}) with $1 + 2A$ in place of $A$. Moreover, (\ref{2.9}) holds for $C^*_n$, $n \ge 1$. By (\ref{2.18}) we see that $\wt{\IP}^G$-a.s. the connected component of $x_0$ in $\{z^* \in \wt{T}^*$; $\sqrt{2} \, \wt{\varphi}_{s(z^*)} > 0\}$ is bounded. This proves that (\ref{2.18}) holds under (\ref{2.3}). 

\medskip
Now, as observed in Remark \ref{rem2.1} 2), (\ref{2.3}) remains true for any choice of the base point $x_0$. Hence, under (\ref{2.3}), $\wt{\IP}^G$-a.s. the connected components of all $x \in T$ in $\{\wt{\varphi} > 0\}$ are bounded. This implies (\ref{2.6}) and concludes the proof of Proposition \ref{prop2.2}. 
\end{proof}

We can now apply the results of \cite{Szni}.
\begin{corollary}\label{cor2.3}
Assume that (\ref{2.3}) holds, then
\begin{equation}\label{2.19}
0 \le h_* \le \sqrt{2u}_* .
\end{equation}

\medskip\n
Moreover, for any $u > 0$, one can couple independent copies $(\varphi_x)_{x \in T}$ and $\cV^u$ of the Gaussian free field on $T$ and the vacant set of random interlacements at level $u$ on $T$, with $(\eta_x)_{x \in T}$ a Gaussian free field on $T$, so that
\begin{equation}\label{2.20}
\mbox{for all $B \subseteq (0,\infty)$, $\{x \in T$; $\eta_x \in \sqrt{2u} + B\} \subseteq \{x \in T$; $\varphi_x \in B\} \cap \cV^u$}.
\end{equation}
\end{corollary}

\begin{proof}
Since $g(x,x) \le R^\infty_x$, cf.~(\ref{1.9}), condition (\ref{2.3}) ensures that (1.43) of \cite{Szni} holds. Moreover, Proposition \ref{prop2.2} shows that condition (\ref{1.32}) of \cite{Szni} holds as well. The claims follow from Corollary 2.5 and Remark 2.6 of \cite{Szni}. Actually, (\ref{2.19}) follows from (\ref{2.6}) alone by the argument of \cite{LyonPere16}, see also (1.33) of \cite{Szni}.
\end{proof}

\begin{remark}\label{rem2.4} \rm ~

\medskip\n
1) Note that an infinite self-avoiding path in $T$ starting at the root $x_0$ only visits the sub-tree $T^\infty$ of vertices with an infinite line of descent. Since $R^\infty_x(T) = R^\infty_x(T^\infty)$ for all $x \in T^\infty$, cf.~(\ref{1.37}), we thus see that
\begin{equation}\label{2.21}
\mbox{condition (\ref{2.3}) holds for $T$ if and only if (\ref{2.3}) holds for $T^\infty$}.
\end{equation}

\medskip\n
2) As the present example shows, it is possible that all terms coincide in (\ref{2.19}). For instance, consider a tree $T$ with root $x_0$, such that $d_x = 1 \vee |x|$, for all $x \in T$. This tree is transient (it contains a binary tree rooted at the descendents of $x_0$). As we now explain, for this tree one has
\begin{equation}\label{2.22}
0 = h_* = \sqrt{2u_*}.
\end{equation}

\n
To this end, note that for all $y \in T$, $R_y^\infty \le R^\infty_{x_0} < \infty$. Thus, one finds that for all $x$ in $T$ (with $c$ a positive constant changing from place to place)
\begin{equation*}
\mbox{\f $\dis\frac{1}{R^\infty_x}$} = \dsl_{y^- = x} \; \mbox{\f $\dis\frac{1}{1 +R^\infty_y}$} \ge d_x \big(1 + R^\infty_{x_0})^{-1} \ge c\, |x|,
\end{equation*}
so that
\begin{equation}\label{2.23}
\mbox{\f $\dis\frac{1}{R^\infty_x (1 + R^\infty_x)}$} \ge c\, |x|.
\end{equation}

\medskip\n
If for $n \ge 1$ we set $C_n = \{x \in T; |x| = n\}$, we obtain a sequence of cut-sets of $T$ with $|C_n| \le n^n$, for which (\ref{2.4}) holds. It then follows that $0 \le h_* \le \sqrt{2u_*}$. Moreover, for any $u > 0$ we have for large $n$
\begin{equation}\label{2.24}
\dsl_{x \in C_n} \; e^{-u \sum_{y \in (x_0,x]} \;\frac{1}{R^\infty_y(1 + R^\infty_y)} } \big(1 - e^{-\frac{u}{R^\infty_x(1 + R^\infty_x)} }\big)^{-1}   \le n^n \,e^{-u  c' n^2} (1 - e^{-u c n})^{-1} \underset{n \r \infty}{\longrightarrow} 0.
\end{equation}

\medskip\n
This implies that $T$ endowed with the weights $c_u(e)$ in (\ref{1.40}) is a recurrent weighted graph (see for instance Corollary 4.2 of \cite{Lyon90}). By (\ref{1.41}), this implies that $\IP^I[x_0 \stackrel{\!\!\!\!\!\!\cV^u}{\longleftrightarrow \infty}] = 0$, and therefore that $u_* \le u$. Since $u > 0$ is arbitrary, we see that $u_* = 0$ and (\ref{2.22}) follows.

\medskip\n
3) In light of the above example one may still wonder whether $\IP^G [x_0 \stackrel{\varphi \ge 0}{\longleftrightarrow} \infty] >0$ holds under (\ref{2.3}). As we now briefly explain this issue is closely connected to a question concerning the geometry of the sign-clusters of $\wt{\varphi}$ the Gaussian free field on the cable system. 

\medskip\n
We first explain how the positivity of $\IP^G [x_0 \stackrel{\varphi \ge 0}{\longleftrightarrow} \infty] $ can be expressed in terms of the sign clusters of $\wt{\varphi}$. Observe that the law of $(sign(\varphi_x))_{x \in T}$ under $\IP^G$ can be generated by first considering the connected components of $\{|\wt{\varphi}|> 0\}$ that meet $T$ (they are $\wt{\IP}^G$-a.s. bounded by (\ref{2.6})), and then by drawing independent symmetric random signs for each of these components. Such a representation can for instance be deduced from the strong Markov property of  $\wt{\varphi}$, combined with an exploration starting from $x_0$ of the successive components of $\{|\wt{\varphi}|> 0\}$ that meet $T$, see also Lemma 3.1 of \cite{LupuA}. 

\medskip\n
Then, denote by $T'$ the random tree obtained by collapsing the sites of $T$ that belong to a same component of $\{|\wt{\varphi}|> 0\}$. From the above representation, the positivity of $\IP^G [x_0 \stackrel{\varphi \ge 0}{\longleftrightarrow} \infty]$ means that with positive $\wt{\IP}^G$-measure, the Bernoulli site percolation with parameter $1/2$ on $T'$ percolates (incidentally, this is indeed the case in the example in 2) above, due to the massive branching of $T'$ in this example). 

\medskip\n
From the above and Theorem 6.2 of \cite{Lyon90}, see also Theorem 5.15 of \cite{LyonPere16}, the positivity of $\IP^G [x_0 \stackrel{\varphi \ge 0}{\longleftrightarrow} \infty] $ now implies that with positive $\wt{\IP}^G$-measure, the so-called branching number of $T'$ (measuring the growth of $T'$) is at least $2$.

\medskip\n
Thus, as a companion to the above question concerning the positivity of $\IP^G [x_0 \stackrel{\varphi \ge 0}{\longleftrightarrow} \infty]$ under (\ref{2.3}), one may also wonder whether under (\ref{2.3})  the branching number of the random tree $T'$ is necessarily bigger or equal to $2$.

 \hfill $\square$
\end{remark}

\section{A sufficient condition for $h_* < \sqrt{2u}_*$}
\setcounter{equation}{0}

In this section, we introduce two conditions, cf.~(\ref{3.1}), (\ref{3.2}) below, and we show in Theorem \ref{theo3.3} that when $u_*$ is non-degenerate these conditions imply that $h_* < \sqrt{2u}_*$. An important step is contained in Proposition \ref{prop3.2}, where an exponential decay of the point to root connection probability in $\{\varphi \ge 0\}$ is derived. As in the previous section, $T$ is a transient tree with root $x_0$, and $T^\infty$ stands for the sub-tree of vertices with an infinite line of descent, cf.~(\ref{1.30}).

\medskip
We now introduce the two conditions mentioned above. The first condition states that
\begin{align}
&\mbox{there exists $A, M, \delta > 0$ such that for large $n$ and all $x \in T^\infty$ with $|x| = n$},\label{3.1}
\\
&\dsl_{y \in (x_0,x)} 1\{R^\infty_y \le A, \, d_{y^-} \le M\} \ge \delta n \nonumber
\intertext{(recall the beginning of Section 1 for notation). The second condition is:}
&\mbox{there exists $B > 0$ such that for large $n$ and all $x \in T^\infty$ with $|x| = n$},\label{3.2}
\\
&\dsl_{y \in (x_0,x]} \; \mbox{\f $\dis\frac{1}{R^\infty_y(1 + R^\infty_y)}$} \le B n. \nonumber
\end{align}

\begin{remark}\label{rem3.1} \rm ~

\medskip\n
1) Note that when (\ref{3.1}) holds, $\{y \in T, R^\infty_y > A\}$ cannot contain an infinite geodesic path in $T$, so that (\ref{3.1}) implies (\ref{2.3}). In particular Corollary \ref{cor2.3} holds as a consequence of (\ref{3.1}).

\bigskip\n
2)  As a result of the observations made above (\ref{1.1}) and below (\ref{1.4}) concerning the effect of moving the location of the root, one sees that when (\ref{3.1}) holds relative to $x_0$, it will also hold relative to any different root $x'_0$ with possibly different $A',M',\delta' > 0$. The same observation applies to condition (\ref{3.2}). \hfill $\square$

\end{remark}

\medskip
Our first main result consists in the derivation of an upper bound showing the exponential decay of $\IP^G[x_0 \stackrel{\varphi \ge 0}{\longleftrightarrow} x]$ for large $x$ in $T^\infty$, when (\ref{3.1}) holds. In essence we will use a strategy of ``entropic repulsion'' to prove this exponential decay, see \cite{Giac03}, p.~13, 14, with however a special twist, see Remark \ref{rem3.3}. 

\begin{proposition}\label{prop3.2}
Assume (\ref{3.1}). Then, there exists $\kappa(A,M,\delta) > 0$ such that
\begin{equation}\label{3.3}
\mbox{for large $n$ and all $x \in T^\infty$ with $|x| = n$, $\IP^G [x_0 \stackrel{\varphi \ge 0}{\longleftrightarrow} x] \le 2 e^{-\kappa n}$}.
\end{equation}
\end{proposition}

\begin{proof}
For $x \in T$ with $|x| = n$, we write $x_0,x_1,\dots,x_n$ for the geodesic path in $T$ from $x_0$ to $x$. Looking at (\ref{3.1}), we see that when $x \in T^\infty$ and $|x|$ is large, one of the two sums corresponding to $y \in (x_0,x)$ with $y = x_k$, $k$ even, or $k$ odd, is at least $\frac{\delta}{2} \,|x|$. Hence, we can find $n_0 \ge 10$, such that for all $x \in T^\infty$ with $|x| \ge n_0$,
\begin{equation}\label{3.4}
\begin{array}{l}
\mbox{there is a subset $I_x \subseteq [x_2,x)$ with $|I_x| \ge \mbox{\f $\dis\frac{\delta}{3}$} \;n$, and $y \not= y' \in I_x \Longrightarrow d(y,y') \ge 2$}, 
\\[1ex]
\mbox{and for each $y \in I_x$, $R^\infty_y \le A$ and $d_{y^-} \le M$}.
\end{array}
\end{equation}
Further, note that for $x \in T^\infty$ with $|x| = n \ge n_0$
\begin{equation}\label{3.5}
\mbox{for all $y \in I_x$ and $y' = (y^-)^-$ (so $|y'| = |y| - 2), R^\infty_{y'} \le A + 2$}.
\end{equation}

\medskip\n
Then, for $x$ in $T^\infty$, with $|x| \ge n_0$, we define the subsets in $[x_0,x]$
\begin{equation}\label{3.6}
J_x = I_x  \cup \{(y^-)^-; y \in I_x\} \; \mbox{and} \; K_x = \{y^-; y \in I_x\}.
\end{equation}

\medskip\n
The sites in $I_x$ are at mutual distance at least $2$, and we thus see that for $x \in T^\infty$, with $|x| = n \ge n_0$
\begin{equation}\label{3.7}
\left\{ \begin{array}{rl}
{\rm i)} & J_x \cap K_x = \phi, \; |J_x| \ge \mbox{\f $\dis\frac{\delta}{3}$}\;n, \; |K_x| \ge \mbox{\f $\dis\frac{\delta}{3}$}\; n,
\\[1ex]
{\rm ii)} & \mbox{for all $y \in J_x$, $R^\infty_y \le A + 2$},
\\[1ex]
{\rm iii)} & \mbox{for all $y \in K_x$, both neighbors of $y$ in $[x_0,x]$ belong to $J_x$},
\\[1ex]
{\rm iv)} & \mbox{for all $y \in K_x$, $d_y \le M$}.
\end{array}\right.
\end{equation}

\medskip\n
We now use a strategy in the spirit of the proof of ``entropic repulsion estimates'' in p.~13, 14 of \cite{Giac03} to bound $\IP^G [x_0 \stackrel{\varphi \ge 0}{\longleftrightarrow} x]$.

\medskip
We consider $x \in T^\infty$ with $|x| = n \ge n_0$ and introduce the event
\begin{equation}\label{3.8}
\mbox{$F_x = \Big\{$for at least  $\mbox{\f $\dis\frac{\delta}{6}$}\;n$ sites $y \in K_x$, $\varphi_{y^-} \le 1$ and $\varphi_{y^+} \le 1\Big\}$},
\end{equation}

\medskip\n
where $y^+$ stands for the only descendant of $y \in K_x$ in $[x_0,x]$ (so $y^-, y^+$ are the two neighbors of $y$ in $[x_0,x]$). We will show that the probabilities of $\{x_0 \stackrel{\varphi \ge 0}{\longleftrightarrow} x\} \cap F_x$ and of $\{x_0 \stackrel{\varphi \ge 0}{\longleftrightarrow} x\} \backslash F_x$ decay exponentially in $|x|$, see (\ref{3.16}) and (\ref{3.22}).  Note that by (\ref{3.7})~iii) the event $F_x$ only depends on the restrictions of $\varphi$ to $J_x$:
\begin{equation}\label{3.9}
F_x \in \sigma_{J_x} \;\mbox{(see (\ref{1.7}) for notation)}.
\end{equation}

\medskip\n
By the Markov property of the Gaussian free field, we find that
\begin{align}
&\mbox{under $\IP^G$, conditionally on $\sigma_{J_x}$, $(\varphi_y)_{y \in K_x}$ are independent and} \label{3.10}
\\
&\mbox{distributed as $a_y + \xi_y$, where} \nonumber
\\[2ex]
&a_y = E_y [H_{J_x} < \infty, \; \varphi(X_{H_{J_x}})], \; \mbox{and} \label{3.11}
\\[2ex]
&\mbox{$\xi_y$ is a centered Gaussian variable with variance} \label{3.12}
\\
&g_{T \backslash \{y^-,y^+\}}(y,y) = \mbox{the effective resistance from $y$ to $\{y^-,y^+\} \cup \{\infty\}$} \nonumber
\\
&\hspace{2.3cm} \stackrel{(\ref{3.7})\, {\rm iv)}}{\ge } (1 + M)^{-1}. \nonumber
\end{align}
We introduce the shorthand notation
\begin{equation}\label{3.13}
A_x = \{\varphi_y \ge 0; \; \mbox{for all} \; y \in [x_0,x]\} \big(= \big\{x_0 \stackrel{\varphi \ge 0}{\longleftrightarrow} x\big\}\big).
\end{equation}

\medskip\n
Then, as an application of (\ref{3.10}) - (\ref{3.12}), we find that
\begin{equation}\label{3.14}
\begin{array}{l}
\IP^G[A_x] = \IP^G[A_x \backslash F_x] + \IP^G [A_x \cap F_x] \;\mbox{with}
\\[1ex]
\IP^G[A_x \cap F_x] \stackrel{(\ref{3.10})-(\ref{3.12})}{\le} \IE^G\Big[F_x, \prod\limits_{y \in K_x} P^{\xi_y}[a_y + \xi_y \ge 0]\Big],
\end{array}
\end{equation}

\medskip\n
where $\xi_y$ has the distribution from (\ref{3.12}) under $P^{\xi_y}$.

\medskip
Note that for each $y \in K_x$ for which $a_y \le 1$ holds, we have
\begin{equation}\label{3.15}
\begin{split}
P^{\xi_y} [a_y + \xi_y \ge 0] & \le P^{\xi_y} [1 + \xi_y \ge 0] = 1 - P^{\xi_y} [\xi_y \le -1]
\\[1ex]
&\!\!\! \stackrel{(\ref{3.12})}{\le} 1 - P^Y[Y \ge \sqrt{1 + M}], \;\mbox{where $Y$ is $N(0,1)$-distributed}.
\end{split}
\end{equation}

\medskip\n
By definition of the event $F_x$ in (\ref{3.8}), on $F_x$ there are at least $\frac{\delta}{6} \;n$ sites $y \in K_x$ such that $a_y \le 1$. Hence, we see that
\begin{equation}\label{3.16}
\IP^G[A_x \cap F_x] \le \big(1 - P^Y[Y \ge   \sqrt{1 + M}]\big)^{\frac{\delta}{6}\,n}.
\end{equation}

\medskip\n
We will now bound $\IP^G [A_x \backslash F_x]$, i.e.~the first term in the right-hand side of the first line of (\ref{3.14}). On $A_x \backslash F_x$, there are at least $|K_x| - \frac{\delta}{6}\;n \ge \frac{\delta}{6} \;n$ sites $y \in K_x$ where $\max\{\varphi_{y^-},\varphi_{y^-}\} \ge 1$ (both $y^-,y^+$ belong to $J_x$, cf.~(\ref{3.7}) iii)). Since on $A_x$, $\varphi \ge 0$ on $[x_0,x]$, we see that
\begin{equation}\label{3.17}
A_x \backslash F_x \subseteq \Big\{\dsl_{y \in J_x} \varphi_y \ge \mbox{\f $\dis\frac{\delta}{12}$} \; n\Big\}
\end{equation}

\n
(two ``consecutive'' $y$ in $K_x$ might share the same neighbor in $J_y$, whence the term $\frac{\delta}{12}\;n$).

\medskip
Note that $\sum_{y \in J_x} \varphi_y$ is a centered Gaussian variable under $\IP^G$, hence we have
\begin{equation}\label{3.18}
\IP^G[A_x \backslash F_x] \stackrel{(\ref{3.17})}{\le} \IP^G \Big[\dsl_{y \in J_x} \varphi_y \ge \mbox{\f $\dis\frac{\delta}{12}$} \; n\Big] \le \exp\Big\{- \mbox{\f $\dis\frac{1}{2}$} \; \Big(\mbox{\f $\dis\frac{\delta}{12}$}\Big)^2 \; \dis\frac{n^2}{{\rm var}(\mbox{\f $\sum\limits_{y \in J_x}$} \varphi_y)} \Big\},
\end{equation}
and we can express the variance in the last term as
\begin{equation}\label{3.19}
\begin{split}
{\rm var}\Big(\dsl_{y \in J_x} \varphi_y\Big) & = \dsl_{y \in J_x} g(y,y) + 2 \dsl_{y< y'\,{\rm in} \, J_x} \,g(y,y')
\\[1ex]
&\!\!\! \stackrel{(\ref{1.11})}{=} \dsl_{y \in J_x} g(y,y) \Big(1 + 2 \dsl_{y' > y, y' \in J_x} \; \pr\limits_{y < y'' \le y'} \alpha_{y''}\Big),
\end{split}
\end{equation}

\n
where the notation $y < y'$ means that $y \in [x_0,y')$ and $y < y'' \le y'$ is defined in a similar fashion. Note that for $y'' \in J_x$, by (\ref{3.7}) ii) we have $R^\infty_{y''} \le A + 2$ and therefore
\begin{equation}\label{3.20}
\alpha_{y''} \le \alpha_A \stackrel{\rm def}{=} \mbox{\f $\dis\frac{A + 2}{A+3}$}, \; \mbox{for all $y'' \in J_x$}.
\end{equation}

\medskip\n
Since $\alpha_{y''} \le 1$, we can restrict the product in the last term of (\ref{3.19}) to the $y'' \in J_x$ and obtain the upper bound
\begin{equation}\label{3.21}
\begin{split}
{\rm var}\Big(\dsl_{y \in J_x} \varphi_y\Big) & \le \dsl_{y \in J_x} g(y,y) \Big(1 + 2 \dsl_{\ell \ge 1} \alpha^\ell_A\Big), \;\mbox{and since $R^\infty_y \le A + 2$ on $J_x$}
\\[1ex]
&\!\!\! \stackrel{(\ref{1.9})}{\le} |J_x| (A+2) \Big(1 + 2 \;\mbox{\f $\dis\frac{\alpha_A}{1 - \alpha_A}$}\Big) \le n(A+2)(2A + 5).
\end{split}
\end{equation}
Coming back to (\ref{3.18}) we find
\begin{equation}\label{3.22}
\IP^G[A_x \backslash F_x] \le \exp\Big\{ - \mbox{\f $\dis\frac{1}{2}$} \; \mbox{\f $\dis\frac{\delta^2}{144}$} \; \mbox{\f $\dis\frac{n}{(A+2)(2A+5)}$}\Big\}.
\end{equation}

\medskip\n
Collecting (\ref{3.16}) and (\ref{3.22}), and coming back to (\ref{3.14}), we see that for all $x \in T^\infty$ with $|x| = n \ge n_0$, we have
\begin{equation}\label{3.23}
\IP^G [x_0 \stackrel{\varphi \ge 0}{\longleftrightarrow} x]  \le\big(1 - P^Y[Y \le - \sqrt{1 + M}]\big)^{\frac{\delta}{6} \,n} + \exp\Big\{ - \mbox{\f $\dis\frac{1}{2}$} \; \mbox{\f $\dis\frac{\delta^2}{144}$} \; \mbox{\f $\dis\frac{n}{(A+2)(2A+5)}$}\Big\}.
\end{equation}
The claim (\ref{3.3}) readily follows.
\end{proof}

\begin{remark}\label{rem3.3} \rm In the first inequality of (\ref{3.18}), it is important that we bound $\IP^G[A_x \backslash F_x]$ in terms of a deviation of $\sum_{J_x} \varphi_y$ and {\it not} of $\sum_{(x_0,x)} \varphi_y$. The point is the following. Whereas the variance of $\sum_{J_x} \varphi_y$ grows at most linearly in $n$, as crucially shown in (\ref{3.21}), the variance of $\sum_{(x_0,x)} \varphi_y$ may grow faster than linearly in $n$, due to the presence of long stretches where, for instance, $d_y = 1$ and $\alpha_y$ is close to $1$. A faster than linear growth of the variance would destroy the exponential decay we obtain in (\ref{3.22}). \hfill $\square$

\end{remark}

We now come to the main result of this section. The proof uses the coupling between the Gaussian free field and random interlacements stated in (\ref{2.20}) of Corollary \ref{cor2.3}.
\begin{theorem}\label{theo3.3}
Assume that $0 < u_* < \infty$ and (\ref{3.1}), (\ref{3.2}) hold, then
\begin{equation}\label{3.24}
(0 \le) \,h_* < \sqrt{2u_*}.
\end{equation}
\end{theorem}

\begin{proof}
In view of (\ref{1.31}), (\ref{1.39}) and (\ref{1.37}), we can assume that $T = T^\infty$. With 
$\kappa$ as in (\ref{3.3}) and $B$ as in (\ref{3.2}), we consider
\begin{equation}\label{3.25}
\mbox{$0 < u < u_*$ where $u = u_* - \rho$ with $\rho = \min\Big(\mbox{\f $\dis\frac{u_*}{2}$}, \mbox{\f $\dis\frac{\kappa}{8B}$}\Big)$}.
\end{equation}
We will show that
\begin{equation}\label{3.26}
h_* \le \sqrt{2u} \;(< \sqrt{2u_*}),
\end{equation}
and the claim (\ref{3.24}) will follow.

\medskip
We use the same notation $\cC^{\cV^u}(x_0)$ as in (\ref{1.34}). Since $u < u_*$, the event
\begin{equation}\label{3.27}
\cP_{x_0,u} = \{x_0 \stackrel{\cV^u}{\longleftrightarrow} \infty\} = \{| \cC^{\cV^u}(x_0)| = \infty\} \; \mbox{has positive $\IP^I$-measure}.
\end{equation}

\n
On the event $\cP_{x_0,u}$, $\cC^{\cV^u}(x_0)$ is an infinite sub-tree of $T$, rooted at $x_0$. However, $\cC^{\cV^u}(x_0)$ is ``thin''. Specifically, if we perform an independent Bernoulli site percolation on the tree $\cC^{\cV^u}(x_0)$ with parameter $p_{x,2\rho}$, for $x \in \cC^{\cV^u}(x_0)$, in the notation of (\ref{1.35}), (\ref{1.36}), the resulting connected component of $x_0$ under the joint law of $\cV^u$ and the above Bernoulli percolation is that of the cluster $\cC^{\cV^{u + 2 \rho}}$. Since $u + 2 \rho > u_*$, cf.~(\ref{3.25}), this cluster is a.s.~finite. As a consequence,
\begin{equation}\label{3.28}
\begin{array}{l}
\mbox{on an event $\wt{\cP}_{x_0,u} \subseteq \cP_{x_0,u}$ with $\IP^I(\cP_{x_0,u} \backslash \wt{\cP}_{x_0,u}) = 0$, a.s.~for the above}
\\
\mbox{Bernoulli site percolation the open cluster of $x_0$ in $\cC^{\cV^u}(x_0)$ is finite}.
\end{array}
\end{equation}

\medskip\n
Let us describe how the proof now proceeds. We will use the above statement expressing the thinness of $\cC^{\cV^u}(x_0)$ when $\wt{\cP}_{x_0,u}$ occurs, to construct a sequence of cut-sets $C_n$ of $\cC^{\cV^u}(x_0)$  tending to infinity, along which $\sum_{x \in C_n} e^{- \frac{3}{4}\kappa |x|}$ tends to $0$, see (\ref{3.32}). By the exponential decay shown in (\ref{3.3}) of Proposition \ref{prop3.2} and a union bound, this will imply that on $\wt{\cP}_{x_0,u}$ the connected component of $x_0$ in $\{\varphi \ge 0\} \cap \cC^{\cV^u}(x_0)$ is $\IP^G$-a.s. finite, see (\ref{3.33}). With the help of the coupling in (\ref{2.20}), it will then be a simple matter to infer that $\IP^G$-a.s., the connected component of $x_0$ in $\{\varphi > \sqrt{2u}\}$ is finite and the claim (\ref{3.26}) will follow. With this strategy in mind, we now proceed with the proof.

\medskip\n
As mentioned in (\ref{1.41}), by Corollary 5.25 of \cite{LyonPere16}, we know that on $\wt{\cP}_{x_0,u}$, 
\begin{equation}\label{3.29}
\begin{array}{l}
\mbox{$\cC^{\cV^u}(x_0)$ endowed with the weights $c_{2 \rho}(e)$, for $e = \{x^-,x\}$, $x \in \cC^{\cV^u}(x_0) \backslash \{x_0\}$},
\\
\mbox{is a recurrent network}.
\end{array}
\end{equation}

\medskip\n
By Corollary 4.2 of \cite{Lyon90}, it then follows that on $\wt{\cP}_{x_0,u}$ for any summable sequence of positive numbers $w_m > 0$, $m \ge 1$, one can find a sequence of cut-sets $C_n, n \ge 1$, of $\cC^{\cV^u}(x_0)$, with $d(x_0,C_n) \underset{n}{\longrightarrow} \infty$, such that
\begin{equation}\label{3.30}
\lim\limits_n \; \dsl_{x \in C_n} w_{|x|} \,e^{-2 \rho \mbox{\f $\sum\limits_{y \in (x_0,x]}$} \; \mbox{\scriptsize $\dis\frac{1}{R^\infty_y(1 + R^\infty_y)}$}} \Big(1 - e^{-2\rho \mbox{\scriptsize $\dis\frac{1}{R^\infty_x(1 + R^\infty_x)}$}}\Big)^{-1} = 0.
\end{equation}
We apply this observation with the choice
\begin{equation}\label{3.31}
w_m = e^{-\frac{\kappa}{2}\,m}, \;m \ge 1 \;\mbox{(and $\kappa$ as in (\ref{3.3}))},
\end{equation}

\medskip\n
and consider the corresponding sequence of cut-sets $C_n$, $n \ge 1$ of $\cC^{\cV^u}(x_0)$. By (\ref{3.2}) and (\ref{3.30}) we see that on $\wt{\cP}_{x_0,u}$
\begin{equation}\label{3.32}
\lim\limits_n \;\dsl_{x \in C_n}\; e^{-\frac{\kappa}{2} |x| - 2 \rho \,B|x|} = 0 \;\mbox{(record for later use that $\frac{\kappa}{2} + 2 \rho \,B \stackrel{(\ref{3.25})}{\le} \frac{3}{4} \; \kappa)$}.
\end{equation}
Thus, on $\wt{\cP}_{x_0,u}$, for large $n$,
\begin{equation}\label{3.33}
\begin{split}
\IP^G[x_0 \stackrel{\varphi \ge 0}{\longleftrightarrow} \infty \;\mbox{in} \; \cC^{\cV^u}(x_0)] & \le \dsl_{x \in C_n} \IP^G[x_0 \stackrel{\varphi \ge 0}{\longleftrightarrow} x] \stackrel{(\ref{3.3})}{\le} 2 \dsl_{x \in C_n} e^{-\kappa |x|} \, \underset{n \r \infty}{\stackrel{(\ref{3.32})}{\longrightarrow}} 0.
\end{split}
\end{equation}
This shows that
\begin{equation}\label{3.34}
\IE^I [\cP_{x_0,u}, \IP^G\big[x_0 \stackrel{\varphi \ge 0}{\longleftrightarrow} \infty \;\mbox{in} \; \cC^{\cV^u}(x_0)]\big] = 0.
\end{equation}

\medskip\n
As observed in Remark \ref{3.1} 1), condition (\ref{3.1}) implies that (\ref{2.20}) holds. Choosing $B = (0,\infty)$ in (\ref{2.20}), we see that
\begin{equation}\label{3.35}
\begin{split}
\IP^G [x_0 \stackrel{\varphi > \sqrt{2u}}{\longleftrightarrow} \infty]& \stackrel{(\ref{2.20})}{\le} \IP^I \otimes \IP^G  [x_0 \stackrel{\cV^u \cap \{\varphi > 0\}}{\longleftrightarrow} \infty]
\\
&\;\; = E^I \big[\cP_{x_0,u}, \IP^G\big[x_0 \stackrel{\varphi \ge 0}{\longleftrightarrow} \infty \;\mbox{in} \; \cC^{\cV^u}(x_0)]\big] \stackrel{(\ref{3.34})}{=} 0.
\end{split}
\end{equation}
This proves that (\ref{3.26}) holds and concludes the proof of the theorem.
\end{proof}

We will later apply Theorem \ref{theo3.3} in Section 5 when $T$ is the typical realization of a super-critical Galton-Watson process conditioned on non-extinction, see Theorem \ref{theo5.4}.

\section{A sufficient condition for $h_* > 0$}
\setcounter{equation}{0}

In this section we provide a sufficient condition which ensures that $h_* > 0$. The main result appears in Proposition \ref{prop4.2}. We provide a first application of the results of this and the previous section in Corollary \ref{cor4.5} and in Remark \ref{rem4.6}.

\medskip
We begin with some notation and a lemma that will be useful in the proof of Proposition \ref{prop4.2}. For $h \in \IR$, we denote by $\pi_h$ the multiplication operator by $1_{[h,\infty)}$, so that for $f$ function on $\IR$, one has
\begin{equation}\label{4.1}
\pi_h \,f = 1_{[h,\infty)} \,f.
\end{equation}
We recall the notation $Q^\alpha$ from (\ref{1.6}).

\begin{lemma}\label{lem4.1}
If $f$ is a bounded, non-decreasing, right-continuous function on $\IR$, which vanishes on $(-\infty,0)$, then
\begin{align}
&\mbox{for any $0 < \alpha \le 1$, $\pi_0\, Q^\alpha f$ has the same properties as noted above for $f$}, \label{4.2}
\\[1ex]
& \mbox{for any $a \in \IR$, $\pi_0 \,Q^\alpha f(a)$ is a non-decreasing function of $\alpha \in (0,1]$}. \label{4.3}
\end{align}
\end{lemma}

\begin{proof}
The claim (\ref{4.2}) is immediate by direct inspection of (\ref{1.6}). We now prove (\ref{4.3}). We can find a positive measure with finite mass, supported on $[0,\infty)$, such that $f(a) = \mu([0,a])$, for $a \ge 0$ (and $f(a) = 0$, for $a < 0$). Hence, we see that for $a \in \IR$,
\begin{equation}\label{4.4}
\begin{split}
Q^\alpha f(a) &\stackrel{(\ref{1.6})}{=} E^Y [f(a \alpha + \sqrt{\alpha}\;Y)]  = E^Y \big[\dil_{[0, \infty)} 1\{a \alpha + \sqrt{\alpha} \;Y \ge b\} \,d \mu(b)\big]
\\[1ex]
& = \dil_{[0, \infty)} P^Y\Big[Y \ge \mbox{\f $\dis\frac{b}{\sqrt{\alpha}} - \sqrt{\alpha}$} \;a\Big] \,d \mu(b).
\end{split}
\end{equation}
Now, when $a \ge 0$, $b \ge 0$, the function $\alpha > 0 \mapsto P^Y[Y \ge \frac{b}{\sqrt{\alpha}} - \sqrt{\alpha} \;a]$ is non-decreasing, and the claim (\ref{4.3}) follows.
\end{proof}

The main result of this section comes next.

\begin{proposition}\label{prop4.2}
Assume that the tree $T$ contains an infinite binary sub-tree $\ov{T}$ rooted at $x_0$, such that for some $M > 0$,
\begin{equation}\label{4.5}
\sup\limits_{x \in \ov{T}} \;d_x\le M.
\end{equation}
Then, $T$ is transient and
\begin{equation}\label{4.6}
h_* > 0.
\end{equation}
\end{proposition}

\begin{remark}\label{rem4.3}  \rm The example in Remark \ref{rem2.4} 2), where one moves the root to its second neighbor, shows that the sole existence of an infinite binary tree rooted at $x_0$ does not guarantee that $h_* > 0$ (in this example we have $0 = h_* = \sqrt{2u}_*$). \hfill $\square$
\end{remark}

\medskip\n
{\it Proof of Proposition \ref{prop4.2}}: The transience of $T$ is immediate. We introduce by analogy with (\ref{1.20}), (\ref{1.21}), $\ov{T}_n = \{x \in \ov{T}$; $|x| = n\} = T_n \cap \ov{T}$, $\ov{B}_n = \{x \in \ov{T}$; $|x| \le n\} = B_n \cap \ov{T}$, as well as the function $\ov{q}\,^n_{x,h}$, for $n \ge 0$, $x \in \ov{B}_n$, $h \in \IR$, via
\begin{equation}\label{4.7}
\left\{ \begin{array}{l}
\ov{q}\,^n_{x,h} = 1_{(-\infty,h)}, \;\mbox{for $x \in \ov{T}_n$},
\\[1ex]
\ov{q}\,^n_{x,h} = 1_{(-\infty,h)} + 1_{[h, \infty)} \pr\limits_{y^- = x, y \in \ov{T}} Q^{\ov{\alpha}}(\ov{q}\,^n_{y,h}), \; \mbox{for $x \in \ov{B}_{n-1}$},
\end{array}\right.
\end{equation}
where
\begin{equation}\label{4.8}
\ov{\alpha} = \mbox{\f $\dis\frac{1}{M+1}$} \ge \alpha_x = \mbox{\f $\dis\frac{R^\infty_x}{1 + R^\infty_x}$}, \;\mbox{for $x \in \ov{T}$}
\end{equation}

\n
(since $R^\infty_x \ge \frac{1}{M}$, by (\ref{1.10}) ii) and (\ref{4.5})).

\medskip
As in Lemma \ref{lem1.2}, we see that for any $x \in \ov{T}$, $h \in \IR$, the functions $\ov{q}\,^n_{x,h}$ increase in $n \ge |x|$ to a function $\ov{q}_{x,h}$, which is non-increasing, $[0,1]$-valued, equal to $1$ on $(-\infty,h)$, with only possible discontinuity at $h$, and such that
\begin{equation}\label{4.9}
\ov{q}_{x,h} = 1_{(-\infty,h)} + 1_{[h, \infty)} \,\pr\limits_{y^- = x, y \in \ov{T}} Q^{\ov{\alpha}}(\ov{q}_{y,h})
\end{equation}

\n
(although we  will not explicitly need it, it is also straightforward to see that $\ov{q}_{x,h} = \ov{q}_{x_0,h}$ for all $x \in \ov{T}$).

\medskip
As we now explain, when $h \ge 0$,
\begin{equation}\label{4.10}
\ov{q}_{x,h}(a) \ge q_{x,h}(a), \;\mbox{for all $a \in \IR$ and $x \in \ov{T}$}.
\end{equation}
Indeed, for any $n \ge 1$,
\begin{equation}\label{4.11}
\ov{q}_{x,h} \ge 1_{(-\infty,h)} = q^n_{x,h}, \;\mbox{when $x \in \ov{T}_n$},
\end{equation}
and for $x \in \ov{B}_{n-1}$,
\begin{equation}\label{4.12}
\begin{split}
q^n_{x,h} & \stackrel{(\ref{1.21})}{=} 1_{(- \infty,h)} + 1_{[h, \infty)} \pr\limits_{y^- = x} Q^{\alpha_y} (q^n_{y,h}) 
\\[1ex]
&\;\; \le 1_{(- \infty,h)} + 1_{[h, \infty)} \pr\limits_{y^- = x, y \in \ov{T}} Q^{\alpha_y} (q^n_{y,h}) 
\\[1ex]
&\;\; \le 1_{(- \infty,h)} + 1_{[h, \infty)} \pr\limits_{y^- = x, y \in \ov{T}} Q^{\ov{\alpha}} (q^n_{y,h}) ,
\end{split}
\end{equation}
where in the last step we have used the fact that $\ov{\alpha} \le \alpha_y$, for $y \in \ov{T}$, see (\ref{4.8}), and applied (\ref{4.3}) of Lemma \ref{lem4.1} to deduce that $\pi_0 \,Q^{\ov{\alpha}}(1 - q^n_{y,h}) \le \pi_0 \, Q^{\alpha_y} (1 - q^n_{y,h})$, whence $\pi_h\,Q^{\alpha_y} (q^n_{y,h}) \le \pi_h \, Q^{\ov{\alpha}}(q^n_{y,h})$ (recall that $h \ge 0$).

\medskip
Combining (\ref{4.11}), with (\ref{4.9}), (\ref{4.12}), we see that the inequality $\ov{q}_{x,h} \ge q^n_{x,h}$ for $x \in \ov{T}_n$ gets propagated to all $x \in \ov{B}_n$. By Lemma \ref{lem1.2}, letting $n$ tend to infinity we obtain (\ref{4.10}).

\medskip
By (\ref{1.26}) we know that
\begin{equation}\label{4.13}
\IP^G[x_0 \stackrel{\varphi \ge h}{\longleftrightarrow} \infty] = 1 - \IE^G[q_{x_0,h}(\varphi_{x_0})].
\end{equation}

\medskip\n
We will show that for small $h > 0$, the right-continuous non-increasing function $\ov{q}_{x_0,h}$ is not identically equal to $1$, so that the same holds for $q_{x_0,h}$ by (\ref{4.10}), and hence the probability in (\ref{4.13}) is positive. In essence, the proof that $\ov{q}_{x_0,h}$ is not identically equal to $1$ will rely on the fact that the largest eigenvalue $\ov{\lambda}_h$ of the Hilbert-Schmidt operator $\ov{L}_h$ defined in (\ref{4.16}) below, is bigger than $1$ for small $h>0$, see (\ref{4.20}). This operator is defined on $L^2(\ov{\nu})$, where $\ov{\nu}$ is a centered Gaussian law on $\IR$, which we now introduce.

\medskip
We denote by $\ov{\nu}$ the centered Gaussian law on $\IR$ with variance
\begin{equation}\label{4.14}
\ov{\sigma}^2 = \dis\frac{\ov{\alpha}}{1 - \ov{\alpha}^2}\,.
\end{equation}
Note that (see also (3.10) - (3.16) of \cite{Szni})
\begin{equation}\label{4.15}
\mbox{$Q^{\ov{\alpha}}$ is a self-adjoint, non-negative Hilbert-Schmidt operator on $L^2(\ov{\nu})$}.
\end{equation}

\medskip\n
We then consider the self-adjoint, non-negative, Hilbert-Schmidt operator on $L^2(\ov{\nu})$ defined by
\begin{equation}\label{4.16}
\ov{L}_h = 2 \pi_h \,Q^{\ov{\alpha}} \pi_h, \;\mbox{for $h \in \IR$},
\end{equation}

\medskip\n
as well as its largest eigenvalue (which coincides with its operator norm)
\begin{equation}\label{4.17}
\ov{\lambda}_h = \| \ov{L}_h\|_{L^2(\ov{\nu}) \r L^2(\ov{\nu})}.
\end{equation}

\n
The proof of Proposition 3.1 of \cite{Szni} shows that
\begin{equation}\label{4.18}
\mbox{$h \r \ov{\lambda}_h$ is a decreasing homeomorphism from $\IR$ onto $(0,2)$}.
\end{equation}

\medskip\n
Moreover, $\pi_0 \,Q^{\ov{\alpha}} \pi_0 (1_{[0,\infty)})(a) = P^Y [\ov{\alpha} \,a + \sqrt{\ov{\alpha}} \;Y \ge 0] >  \frac{1}{2}$ for $a > 0$, so that
\begin{equation}\label{4.19}
\ov{\lambda}_0  \ge \big(\ov{L}_0 \,1_{[0,\infty)}, 1_{[0,\infty)}\big)_{L^2 ( \ov{\nu})} / \|1_{[0,\infty)}\|^2_{L^2(\ov{\nu})} > 1.
\end{equation}
It thus follows that
\begin{equation}\label{4.20}
\mbox{$\ov{\lambda}_h > 1$ for small $h > 0$}.
\end{equation}

\n
One can then proceed as in the proof of (\ref{3.33}) of \cite{Szni} to deduce that
\begin{equation}\label{4.21}
\ov{\gamma}_h = 1 - \dis\int^\infty_h d \,\ov{\nu} (a) \, \ov{q}_{x_0,h}(a) > 0, \; \mbox{when $\ov{\lambda}_h > 1$}.
\end{equation}

\n
This shows that for small $h > 0$, $\ov{q}_{x_0,h}$ and hence $q_{x_0,h}$ by (\ref{4.10}), are not identically equal to $1$. As explained below (\ref{4.13}), it follows that $\IP^G [x_0 \stackrel{\varphi \ge h}{\longleftrightarrow} \infty] > 0$ for small $h > 0$, and hence $h_* > 0$. This proves (\ref{4.6}). \hfill $\square$

\begin{remark}\label{rem4.4} \rm When $M$ in (\ref{4.5}) is chosen as an integer bigger or equal to $2$, the quantity $\ov{\gamma}_h$ in (\ref{4.21}) can be interpreted as the probability that the Gaussian free field $\ov{\varphi}$ on an $(M+2)$-regular tree, when restricted to a given ``forward binary tree'' rooted at some point, is such that the root belongs to an infinite connected component of $\{\ov{\varphi} \ge h\}$ (see also Section 3 of \cite{Szni}). \hfill $\square$
\end{remark}

\medskip
As an application of Theorem \ref{theo3.3} and Proposition \ref{prop4.2} we have
\begin{corollary}\label{cor4.5}
Assume that the rooted tree $T$ has bounded degree, contains an infinite binary sub-tree, and $\sup_{x \in T} R^\infty_x < \infty$, then
\begin{equation}\label{4.22}
0 < h_* < \sqrt{2u}_* < \infty.
\end{equation}
\end{corollary}

\begin{proof}
By the observation below (\ref{1.4}), moving the root of $T$ if necessary, we can assume that the infinite binary tree is rooted at $x_0$, the root of $T$. The $R^\infty_x$, $x \in T$, are uniformly bounded, and bounded away from $0$ by (\ref{1.10}) ii). It is now straightforward from (\ref{1.34}), (\ref{1.36}) to infer that $0 < u_* < \infty$, see also Theorem 5.1 of \cite{Teix09b}. Conditions (\ref{3.1}) and (\ref{3.2}) are immediate and the assumptions of Proposition \ref{prop4.2} are fulfilled as well. The claim (\ref{4.22}) now follows from Theorem \ref{theo3.3} and Proposition \ref{prop4.2}. 
\end{proof}

\begin{remark}\label{rem4.6} \rm 
In particular, when $T$ is an infinite tree of bounded degree such that outside a finite set, all vertices have degree at least $3$, the assumptions of the above corollary are fulfilled by $T^\infty$, and in view of (\ref{1.31}), (\ref{1.39}) one has
\begin{equation}\label{4.23}
0 < h_* < \sqrt{2u}_* < \infty.
\end{equation}
\hfill $\square$
\end{remark}

\section{Application to super-critical Galton-Watson trees}
\setcounter{equation}{0}

We will now apply the results of the previous sections to the case where $T$ is a super-critical Galton-Watson tree conditioned on non-extinction, and the root $x_0$ is the initial ancestor. Our main results appear in Theorems \ref{theo5.4} and \ref{theo5.8}. An important step is carried out in Proposition \ref{prop5.2} where it is checked that condition (\ref{3.1}) holds almost surely. Theorem \ref{theo5.4} establishes the inequality $h_* < \sqrt{2u}_*$ and the exponential decay in $|x|$ of $\IP^G[x_0 \stackrel{\varphi \ge 0}{\longleftrightarrow} x]$ in a broad enough generality. Theorem \ref{theo5.8} provides a sufficient condition for $h_* > 0$.

\medskip
We first introduce some notation and recall various known facts. We denote by $\nu$ the offspring distribution and assume that
\begin{equation}\label{5.1}
m = \dsl^\infty_{k=0} k \, \nu(k) \in (1,\infty)
\end{equation}

\n
(we are in the super-critical regime). We denote by $P^\nu$ the probability governing the Galton-Watson tree and by $f$ the generating function of $\nu$, which is strictly convex and equals 
\begin{equation}\label{5.2}
f(s) = \dsl^\infty_{k=0} s^k \nu (k), \; 0 \le s \le 1 .
\end{equation}

\n
If $q$ stands for the extinction probability $P^\nu[|T| < \infty]$, then $q \in [0,1)$ is the smallest fixed point of $f$ on $[0,1]$, the only other fixed point being $s = 1$. Moreover, conditioned on non-extinction, i.e.~under $P_*^\nu[ \cdot ] = P^\nu[ \cdot \big| \,|T| = \infty]$, the sub-tree $T^\infty$ of sites with an infinite line of descent corresponds to a Galton-Watson tree with offspring distribution $\nu_\infty$ having the generating function
\begin{equation}\label{5.3}
f_\infty(s) = \mbox{\f $\dis\frac{f(q+s(1-q))-q}{1-q}$} , \; 0 \le s \le 1,
\end{equation}

\medskip\n
so that $\nu_\infty(0) = 0$, and $\nu_\infty$ has same mean $m$ as $\nu$, see for instance Proposition 5.28 of \cite{LyonPere16}. In addition, one knows that $P^\nu_*$-a.s., $T$ (and $T^\infty$) are transient, see Theorem 3.5 and Corollary 5.10 of \cite{LyonPere16}.

\medskip
It is further known from Section 3 of \cite{Tass10}, that $u_*$ is $P_*^\nu$-a.s. constant and by Theorem 1 and (1.5)' of \cite{Tass10}, that
\begin{equation}\label{5.4}
\begin{array}{l}
\mbox{$u_* \in (0,\infty)$ is characterized as the unique solution of $f'_\infty(\cL(u)) = 1$, where}
\\
\mbox{$\cL(u) = E^\nu_*[e^{-u(1 - \alpha_{x_0})}] (= E^{\nu_\infty}[e^{-u(1-\alpha_{x_0})}])$, for $u \ge 0$}
\end{array}
\end{equation}

\n
($E^\nu_*$, $E^{\nu_\infty}$ stand for the respective $P^\nu_*$ and $P^{\nu_\infty}$-expectations).

\medskip
As we now explain, conditioned on non-extinction, $h_*$ is almost surely constant. We will later see, cf.~(\ref{5.8}), that this constant is finite.

\begin{lemma}\label{lem5.1}
$h_*(T) \ge 0$ is almost everywhere constant under $P^\nu_* + P^{\nu_\infty}$.
\end{lemma}

\begin{proof}
By (\ref{1.31}) and the observation below (\ref{5.2}), $h_*(T)$ has same law under $P^\nu_*$ and $P^{\nu_\infty}$. We thus only need to consider $P^\nu_*$. For $h \in \IR$, we introduce the event (see Remark \ref{rem1.3} for the notation)
\begin{equation*}
\mbox{$A_h = \{T$; $q^T_{x_0,h}$ is identically $1\}$}.
\end{equation*}

\medskip\n
This event (in the space of Galton-Watson trees) is hereditary (or in the terminology of \cite{LyonPere16}, p.~136, the property it describes is inherited). Namely, all finite Galton-Watson trees belong to $A_h$ (by (\ref{1.24})), and when $T \in A_h$, the equation
\begin{equation*}
q^T_{x_0,h} \stackrel{(\ref{1.23})}{=} 1_{(-\infty,h)} + 1_{[h,\infty)} \pr\limits_{|x|=1} Q^{\alpha_x} (q^T_{x,h})
\end{equation*}

\n
implies that $q^T_{x,h}$ are identically equal to $1$ for all $|x| = 1$ in $T$, and by (\ref{1.29}) we see that $T_x \in A_h$ for all  $|x| = 1$ in $T$. It now follows by Proposition 5.6 of \cite{LyonPere16} that $P^\nu_*[A_h] \in \{0,1\}$ for all $h \in \IR$.

\medskip
By (\ref{1.27}) we know that for any $h$ in $\IR$
\begin{equation}\label{5.5}
\mbox{$P^\nu_*$-a.s., $\{h_*(T) < h\} \subseteq A_h$ and $\{h_*(T) > h\} \subseteq A^c_h$}.
\end{equation}

\n
The random variable $h_*(T)$ is non-negative, see (\ref{1.19}), so that $P^\nu_*[A^c_h] = 1$ for all $h < 0$. If one sets
\begin{equation}\label{5.6}
\mbox{$h_0 = \inf\{h \in \IR; P^\nu_* [A_h] = 1\}$ (with the convention $\inf \phi = \infty$)},
\end{equation}

\n
it is now routine to see with (\ref{5.5}) and the above $0$-$1$ law that $P^\nu_*$-a.s., $h_*(T) = h_0$. This proves Lemma \ref{lem5.1}.
\end{proof}

\medskip
We will now see that condition (\ref{3.1}) is $P^\nu_*$-a.s.~fulfilled. In particular, this implies that $P^\nu_*$-a.s. the conclusions of Corollary \ref{cor2.3} (see Remark \ref{rem3.1} 1)) and of Proposition \ref{prop3.2} hold.

\begin{proposition}\label{prop5.2}
There exist $A,M,\delta > 0$ such that $P^\nu_*$-a.s., for large $n$ and all $x \in T^\infty$ with $|x| =n$,
\begin{equation}\label{5.7}
\dsl_{y \in (x_0,x)} 1\{R^\infty_y \le A, d_{y^-} \le M\} \ge \delta n .
\end{equation}
In particular, the $P_*^\nu$-a.s. constant quantities $h_*$ and $u_*$ satisfy
\begin{equation}\label{5.8}
0 \le h_* \le \sqrt{2u}_* < \infty .
\end{equation}
\end{proposition}

\begin{proof}
Once we prove (\ref{5.7}), the second claim is by Remark \ref{rem3.1} 1) a direct consequence of Corollary \ref{cor2.3}, together with Lemma \ref{lem5.1} and (\ref{5.4}). We thus only need to prove (\ref{5.7}), which pertains to $T^\infty$. By the observation below (\ref{5.2}) we can work with $P^{\nu_\infty}$ in place of $P^\nu_*$, so that $T = T^\infty$, $P^{\nu_\infty}$-a.s.. We are going to show that
\begin{align}
&\mbox{for all $\eta > 0$, there exists $M > 0$ such that $P^{\nu_\infty}$-a.s., for large $n$ and all}\label{5.9} \\
&\mbox{$x \in T$ with $|x| = n$, $\dsl_{y \in [x_0,x)} 1\{d_y \ge M\} \le \eta n$}, \nonumber
\\[2ex]
&\mbox{there exists $\eta' \in (0,1)$ and $A > 0$ such that $P^{\nu_\infty}$-a.s., for large $n$ and all}\label{5.10} \\
&\mbox{$x \in T$ with $|x| = n$, $\dsl_{y \in (x_0,x)} 1\{R^\infty_y \le A\} \ge (1- \eta') n$}. \nonumber\end{align}

\n
Then, choosing $\delta,\eta > 0$ such that $\eta' + \eta + \delta <1$, it will follow that (\ref{5.7}) holds with $P^{\nu_\infty}$ in place of $P^\nu_*$ and, as mentioned above, this will prove our claim.

\medskip
We start with the proof of (\ref{5.9}). We denote by $\ov{P}$ the measure on Galton-Watson trees endowed with a spine denoted by $w_i$, $i \ge 0$ (so $w_0 = x_0$ and $w^-_{i+1} = w_i$, for each $i \ge 0$), such that under $\ov{P}$ the individuals on the spine reproduce with the size-biased distribution (recall that $\nu$ and $\nu_\infty$ have same mean $m$)
\begin{equation}\label{5.11}
\ov{\nu}(k) = \mbox{\f $\dis\frac{k}{m}$} \;\nu_\infty(k), \; k \ge 1,
\end{equation}

\n
the individuals off the spine reproduce with distribution $\nu_\infty$, and at each step, the next element $w_{i+1}$ on the spine is chosen uniformly among the offspring of $w_i$ (we refer to Chapter 1 \S 3 of \cite{AbraDelm15}, or to Chapter 12 \S 1 of \cite{LyonPere16} for details).

\medskip
Then, for $M, \eta, \lambda > 0$ and $n \ge 1$, we have
\begin{equation}\label{5.12}
\begin{array}{l}
\mbox{$P^{\nu_\infty} \big[$for some $x$ in $T$ with $|x| = n$, $\dsl_{y \in [x_0,x)} 1\{d_y \ge M\} > \eta n\big] \le$}
\\[3ex]
E^{\nu_\infty} \big[\dsl_{|x| = n} 1\big\{\dsl_{y \in [x_0,x)} 1\{d_y \ge M\} > \eta n\big\}\big] =
\\[-1ex]
m^n \,\ov{P}\big[\dsl^{n-1}_{i=0} 1\{d_{w_i} \ge M\} > \eta n\big]\!\!\! \stackrel{\begin{array}{c}
\mbox{\scriptsize exponential}\\[-1ex]
\mbox{\scriptsize Chebyshev}
\end{array}}{\le} \!\!\!\!\! m^n e^{-\lambda \eta n}  \;\ov{E}\big[e^{\lambda \sum\limits^{n-1}_{i=0} 1\{d_{w_i} \ge M\}}\big]
\end{array}
\end{equation}

\n
and we made use of Lemma 1.3.2 of \cite{AbraDelm15} for the equality on the third line. The $d_{w_i}$, $i \ge 0$, are i.i.d. and distributed as $\ov{\nu}$ under $\ov{P}$. Hence, we have ($\ov{E}$ stands for the $\ov{P}$-expectation):
\begin{equation}\label{5.13}
\ov{E}\big[e^{\lambda \sum\limits^{n-1}_{i=0} 1\{d_{w_i} \ge M\}}\big] = \ov{E}[e^{\lambda \,
1_{\{d_{x_0} \ge M\}}}]^n \stackrel{(\ref{5.11})}{=} m^{-n} \big(m  + (e^\lambda - 1) \,E^{\nu_\infty} [d_{x_0},d_{x_0} \ge M]\big)^n .
\end{equation}

\medskip
If we now choose $\lambda_0 > 0$ such that
\begin{equation}\label{5.14}
e^{\frac{\lambda_0}{2} \,\eta} > m,
\end{equation}
and $M_0$ (large) such that
\begin{equation}\label{5.15}
e^{\frac{\lambda_0}{2} \,\eta} > m + (e^{\lambda_0} - 1) \,E^{\nu_\infty} [d_{x_0}, d_{x_0} \ge M_0],
\end{equation}

\n
we find after insertion in the last line of (\ref{5.12}) that
\begin{equation}\label{5.16}
\mbox{$P^{\nu_\infty}\Big[$for some $x$ in $T$ with $|x| = n$, $\dsl_{y \in [x_0,x)}1\{d_y \ge M_0\} > \eta n\Big] \le e^{-\frac{\lambda_0}{2} \,\eta n}$}.
\end{equation}
The claim (\ref{5.9}) follows by Borel-Cantelli's lemma.

\medskip
We then turn to the proof of (\ref{5.10}).  We note that if $x \in T$ and $y \in (x_0,x)$, we can set 
\begin{equation}\label{5.17}
\wt{R}_y = \left\{ \begin{array}{l}
\infty, \; \mbox{if $d_y = 1$,}
\\[1ex]
\mbox{$\min\{1 + R^\infty_{y'}; y' \in T_y$ and $y' \notin [y,x]\}$ otherwise}\,.
\end{array}\right.
\end{equation}

\n
Then, clearly $R^\infty_y \le \wt{R}_y$. If we now choose $A$ large enough such that $P^{\nu_\infty} [1 + R^\infty_{x_0} \le A] > 0$, it follows from (\ref{2.16}) in Lemma 1 of \cite{GrimKest}, that for some $\eta' \in (0,1)$
\begin{equation}\label{5.18}
P^{\nu_\infty} \Big[\min\limits_{|x| = n} \;\dsl_{y \in (x_0,x)} 1\{\wt{R}_y \le A\} \le (1- \eta')\,n\Big] \;\mbox{decays geometrically in $n$}.
\end{equation}

\n
By Borel Cantelli's lemma, we see that
\begin{equation}\label{5.19}
\begin{array}{l}
\mbox{$P^{\nu_\infty}$-a.s. for large $n$ and all $x \in T$ with $|x| = n$,}
\\[1ex]
\dsl_{y \in (x_0,x)} 1\{R^\infty_y \le A\} \ge \dsl_{y \in (x_0,x)} \;1\{\wt{R}_y \le A\} \ge (1- \eta')\,n\,.
\end{array}
\end{equation}
The claim (\ref{5.10}) follows. This completes the proof of Proposition \ref{prop5.2}. 
\end{proof}

We will now see that the existence of some finite exponential moment of the offspring distribution $\nu$ ensures that (\ref{3.2}) holds $P^\nu_*$-almost surely. This is the last step before Theorem \ref{theo5.4}, which is in essence an application of Theorem \ref{theo3.3} to the present set-up.

\medskip
We introduce the condition
\begin{equation}\label{5.32}
\mbox{for some $\gamma > 0$, $\dsl_{k \ge 0} e^{\gamma k} \nu(k) < \infty$}.
\end{equation}

\n
Under (\ref{5.32}) the generating function $f$ in (\ref{5.2}) has an analytic extension to a disc in the complex plane centered at the origin with radius bigger than $1$. In view of (\ref{5.3}) a similar property holds for $f_\infty$, and hence
\begin{equation}\label{5.33}
\mbox{for some $\gamma_\infty > 0$, $\dsl_{k \ge 1} e^{\gamma_\infty k} \nu_\infty(k) < \infty$}.
\end{equation}

\begin{proposition}\label{prop5.3}
When (\ref{5.32}) holds, then there exists $B > 0$, such that $P^{\nu}_*$-a.s., for large $n$ and all $x \in T^\infty$ with $|x| = n$,
\begin{equation}\label{5.34}
\dsl_{y \in (x_0,x]} \; \mbox{\f $\dis\frac{1}{R^\infty_y (1 + R^\infty_y)}$} \le B \, n.
\end{equation}
\end{proposition}

\begin{proof}
As in the proof of Proposition \ref{prop5.2}, we can work with $P^{\nu_\infty}$ in place of $P^\nu_*$, so that $P^{\nu_\infty}$-a.s., $T = T^\infty$. We first observe that for $B > 0$ and $n \ge 1$, 
\begin{equation}\label{5.35}
\begin{split}
P^{\nu_\infty}\big[\mbox{for some $x$ with $|x| = n$}, d_x \ge \mbox{\scriptsize $\dis\frac{B}{2}$} \;n\big] & \le m^n \,P^{\nu_\infty}\big[d_{x_0} \ge \mbox{\scriptsize $\dis\frac{B}{2}$} \;n\big]
\\
& \le m^n \,e^{-\gamma_\infty \,\frac{B}{2} \,n} \,E^{\nu_\infty}[e^{\gamma_\infty d_{x_0}}].
\end{split}
\end{equation}

\n
Moreover, with $\ov{P}$ as below (\ref{5.10}), and $w_i, i \ge 0$, the spine, we have
\begin{equation}\label{5.36}
\begin{array}{l}
P^{\nu_\infty}\big[\mbox{for some $x$ with $|x| = n$}, \dsl_{y \in [x_0,x)} d_y \ge \mbox{\f $\dis\frac{B}{2}$} \;n\big]  \le
\\[1ex]
E^{\nu_\infty} \big[ \dsl_{|x|=n} 1\big\{\dsl_{y \in [x_0,x)} d_y \ge \mbox{\scriptsize$\dis\frac{B}{2}$} \;n\big\}\big] = m^n  \, \ov{P} \big[\dsl^{n-1}_{i=0} d_{w_i} \ge \mbox{\scriptsize $\dis\frac{B}{2}$} \;n\big] .
\end{array}
\end{equation}

\n
Under $\ov{P}$ the variables $d_{w_i}, i \ge 0$, are i.i.d. $\ov{\nu}$-distributed, with $\ov{\nu}$ the size-biased distribution in (\ref{5.11}). By the exponential Chernov bound, we find that 
\begin{equation}\label{5.37}
\begin{array}{l}
\ov{P} \big[\dsl^{n-1}_{i=0} d_{w_i} \ge \mbox{\scriptsize$\dis\frac{B}{2}$} \;n\big] \le \exp \{- n \,\ov{I} \big(\mbox{\scriptsize$\dis\frac{B}{2}$} \big)\big\}, \; \mbox{where for $a \ge 0$},
\\[3ex]
\ov{I}(a) = \sup\limits_{\lambda \ge 0} \{\lambda \,a - \log \ov{E}[e^{\lambda d_{x_0}}]\} \ge \mbox{\f $\dis\frac{\gamma_\infty}{2}$} \;a - b,
\\[3ex]
\mbox{with $b= \log \ov{E} [e^{\frac{\gamma_\infty}{2}\, d_{x_0}}]$ ($< \infty$ by (\ref{5.33}))}.
\end{array}
\end{equation}

\medskip\n
If we now choose $B_0$ large enough so that $\frac{\gamma_\infty}{2} \,B_0 > \log m$ and $\ov{I} (\frac{B_0}{2}) > \log m$, then we see that the probabilities in (\ref{5.35}) and (\ref{5.36}) are summable in $n$. Hence, by Borel-Cantelli's lemma, we find that $P^{\nu_\infty}$-a.s., for large $n$ and all $x \in T$ with $|x| = n$,
\begin{equation}\label{5.38}
\dsl_{y \in (x_0,x]} d_y \le B_0 \,n.
\end{equation}

\n
Since $(R^\infty_y (1 + R^\infty_y))^{-1} \le d_y$ by (\ref{1.10}) ii), this proves Proposition \ref{prop5.3}.
\end{proof}

We now come to one of the main results of this section, which states that for a super-critical Galton-Watson tree conditioned on non-extinction if the offspring distribution has some finite exponential moment then $h_* < \sqrt{2u}_*$. More precisely, one has

\begin{theorem}\label{theo5.4}
The deterministic critical values $h_*$ and $u_*$ attached to a super-critical Galton-Watson tree conditioned on non-extinction, for which the offspring distribution satisfies (\ref{5.1}), (\ref{5.32}), are such that 
\begin{equation}\label{5.39}
0 \le h_* < \sqrt{2 u}_* < \infty.
\end{equation}

\n
In addition, under (\ref{5.1}) alone, there exists $\beta > 0$ such that $P^\nu_*$-a.s.,
\begin{equation}\label{5.40}
\mbox{for large $n$ and all $x \in T$ with $|x| = n$, $\IP^G[x_0 \stackrel{\varphi \ge 0}{\longleftrightarrow} x] \le e^{-\beta n}$}.
\end{equation}
\end{theorem}

\begin{proof}
As recalled in (\ref{5.4}), we know that $0 < u_* < \infty$. Moreover, by Propositions \ref{prop5.2} and \ref{prop5.3}, we see that $P^\nu_*$-a.s., the conditions (\ref{3.1}) and (\ref{3.2}) are fulfilled. The claim (\ref{5.39}) now follows from Theorem \ref{theo3.3}. 

\medskip
Let us now prove (\ref{5.40}). By Propositions \ref{prop5.2} and \ref{prop3.2} we know that (\ref{5.40}) holds with $T$ replaced by $T^\infty$ and $e^{-\beta n}$ by $2 e^{-\kappa n}$. When the extinction probability $q \in [0,1)$ vanishes, (\ref{5.40}) readily follows. Otherwise, we set $m' = f'(q) < 1$ and choose $\eta \in (0,1)$ so that $\ov{m} = m'^{(1- \eta)} m^\eta < 1$. For $n \ge 1$, we set $\wt{n} = [\eta n]$ and for $x \in T$ such that $|x| = n$, we write $\wt{x}$ for the site in $[x_0,x]$ such that $|\wt{x}| = \wt{n}$. Clearly, $\IP^G[x_0 \stackrel{\varphi \ge 0}{\longleftrightarrow} x] \le \IP^G[x_0 \stackrel{\varphi \ge 0}{\longleftrightarrow} \wt{x}]$, and by the above mentioned exponential decay along $T^\infty$, the claim (\ref{5.40}) will follow (with say $\beta = \frac{1}{2} \,\eta \kappa)$ once we show that
\begin{equation}\label{5.41}
\mbox{$P^\nu_*$-a.s. for large $n$ and all $x \in T$ with $|x| = n$, $\wt{x} \in T^\infty$}.
\end{equation}

\n
To see this last point, we note that the Galton-Watson tree conditioned on extinction has mean offspring $m'$, see Lemma 1.2.5 of \cite{AbraDelm15} or Proposition 5.28 of \cite{LyonPere16}, and
\begin{equation*}
\begin{array}{l}
\mbox{$P^\nu[$for some $x \in T$ with $|x| = n$, $\wt{x} \notin T^\infty] \le$}
\\[1ex]
E^\nu \big[\dsl_{|y| = \wt{n}} \big(1\{|T_y| < \infty\big\} \;\dsl_{|x| = n} \,1 \{x \in T_y\}\big)\big] = m^{\wt{n}} q\,m'\,^{\!n-\wt{n}} \le q \,\ov{m}\,^n .
\end{array}
\end{equation*}
This last geometric series is convergent and (\ref{5.41}) follows by Borel Cantelli's lemma. This concludes the proof of Theorem \ref{theo5.4}.
\end{proof}

We will now introduce a sufficient condition which implies that $h_* > 0$ for the super-critical Galton-Watson tree conditioned on non-extinction under consideration. The argument we use is somewhat in the spirit of \cite{Tass10} for random interlacements on Galton-Watson trees, but the situation is more complicated in the case of the Gaussian free field.

\begin{theorem}\label{theo5.8}
Assume that the mean $m$ of the offspring distribution $\nu$ satisfies
\begin{equation}\label{5.54}
 m >2,
\end{equation}
then, almost surely on non-extinction, i.e.~$P^\nu_*$-a.s.,
\begin{equation}\label{5.56}
h_* > 0.
\end{equation}
\end{theorem}

\begin{proof}
By (\ref{1.31}) it suffices to consider $T^\infty$ in place of $T$, and by the observation below (\ref{5.2}) to replace $\nu$ by $\nu_\infty$, so that $P^{\nu_\infty}$-a.s., $T^\infty = T$, noting that $\nu_\infty$ has the same mean $m$ as $\nu$.

\medskip
In essence, the strategy of the proof is to obtain for small $h>0$, uniformly in $n$, a positive lower bound for the quantity $\gamma^n_h$ defined below (\ref{5.58a}), which involves the function $r^n_{x_0,h}(\cdot)$. From this lower bound we will infer that $P^{\nu_\infty}$-a.s., the function $q^n_{x_0,h}(\cdot)$ is not identically equal to $1$, so that $h_* \ge h (> 0)$. The above mentioned lower bound will stem from an inequality based on the branching property (see (\ref{5.4}) for the definition of $\cL(\cdot)$)

\begin{equation*}
m' \gamma^n_h > f_{\infty,M} \big(\cL(\mbox{\f $\dis\frac{h^2}{2}$})\big) - f_{\infty,M} \big(\cL(\mbox{\f $\dis\frac{h^2}{2}$}) - \gamma^n_h\big)
\end{equation*}
where $m' \in (2,m)$, and $h>0$ is such that the derivative of the truncated generating function $f_{\infty,M}$ (see (\ref{5.57a}))  is bigger than $m'$ in a neighborhood of $\cL(h^2/2)$.

\medskip
Let us now proceed with the proof. Our first objective is to establish the key identity that appears in (\ref{5.59}) below. For $h \in \IR$, $n \ge 0$, and $x$ in $T$ with $|x| \le n$ (from now on we assume that $T^\infty = T$), we consider in the notation of (\ref{1.21}) the function $r^n_{x,h}(a) = 1- q^n_{x,h}(a)$, for $a \in \IR$, so that
\begin{equation*}
\begin{array}{l}
\left\{ 
\begin{array}{l}
r^n_{x,h} = 1_{[h,\infty)}, \; \mbox{for $x \in T_n$},
\\[1ex]
r^n_{x,h} = 1_{[h,\infty)} \big(1 - \textstyle{\prod\limits_{y^- = x}} \big(1 - Q^{\alpha_y}(r^n_{y,h})\big)\big), \;\mbox{for $x$ in $T$, with $|x| < n$}
\end{array}
\right.
\\
\quad \mbox{(and $r^n_{x,h}$ decreases with $n$ and tends to $r_{x,h} = 1-q_{x,h}$, as $n \r \infty$,}
\\
\quad \; \mbox{by Lemma \ref{lem1.2}).}
\end{array}
\end{equation*}

\n
Multiplying both members of the last equality by $e^{-\frac{a^2}{2R^\infty_{x}}} = \Pi_{y^- = x}\; e^{-\frac{a^2}{2(1+ R^\infty_y)}}$, we obtain for $x$ in $T$ with $|x| < n$ and $a$ in $\IR$,
\begin{equation}\label{5.57}
e^{-\frac{a^2}{2R^\infty_{x}}}  r^n_{x,h} (a) = 1_{[h,\infty)}(a) \big( \textstyle{\prod\limits_{y^- = x}} \!\! e^{-\frac{a^2}{2(1+ R^\infty_y)}} \!\! - \!\!\textstyle{\prod\limits_{y^- = x}} \!\!\! \big(e^{-\frac{a^2}{2(1+ R^\infty_y)}}  \!\! - \! e^{-\frac{a^2}{2(1+ R^\infty_y)}} Q^{\alpha_y} (r^n_{y,h})(a)\big)\big).
\end{equation}

\n
Recall that $(1 + R^\infty_{x_0})^{-1} = 1 - \alpha_{x_0}$ and $\cL(u) = E^{\nu_\infty} [e^{-u(1 - \alpha_{x_0})}]$ in the notation of (\ref{5.4}). 
For $M \ge 1$, we write
\begin{equation}\label{5.57a}
f_{\infty,M} (s) = \dsl_{k \le M}  s^k \nu_\infty(k) \; , 0 \le s \le 1.
\end{equation}

Choosing $x = x_0$, and $n+1$ in place of $n$ in (\ref{5.57}), and multiplying both members by the indicator function of the event $\{ d_{x_0} \le M \}$, after $P^{\nu_\infty}$-integration and using the branching property, we see that for $h,a \in \IR$, $n \ge 0$, and $M \ge 1$
\begin{equation}\label{5.58}
\begin{array}{l}
E^{\nu_\infty} \big[e^{-\frac{a^2}{2R^\infty_{x_0}}} \; r^{n+1}_{x_0,h}(a), d_{x_0} \le M  \big] = 
\\[1ex]
1_{[h,\infty)}(a) \big\{f_{\infty,M} \big(\cL (\mbox{\f $\dis\frac{a^2}{2}$})\big) - f_{\infty,M}\big(\cL (\mbox{\f $\dis\frac{a^2}{2}$}) - E^{\nu_\infty} [ e^{-\frac{a^2}{2(1+R^\infty_{x_0})}} Q^{\alpha_{x_0}} (r^n_{x_0,h})(a)]\big)\big\}.
\end{array}
\end{equation}

By  (\ref{5.54}) we can now choose a large $M$, $s_0 \in [0,1)$, and $m' \in (2,m)$ such that
\begin{equation}\label{5.58a}
f'_{\infty,M}(s) > m' ,\; \mbox{for} \; s_0 \le s \le 1 .
\end{equation}

\n
With this choice of $M$, when $a = h \ge 0$, we denote by $\wt{\gamma}^n_h$ the expectation on the left-hand side of (\ref{5.58}) and by $\gamma^n_h$ the expectation on the right-hand side so that
\begin{equation}\label{5.59}
\wt{\gamma}^n_h = f_{\infty,M} \big(\cL(\mbox{\f $\dis\frac{h^2}{2}$})\big) - f_{\infty,M} \big(\cL(\mbox{\f $\dis\frac{h^2}{2}$}) - \gamma^n_h\big), \;\mbox{for $n \ge 0$ and $h \ge 0$}.
\end{equation}
In addition (since $r^n_{x_0,h}(\cdot)$ is non-decreasing)
\begin{equation}\label{5.60}
\begin{split}
\gamma^n_h &\ge  e^{-\frac{h^2}{2}} E^{\nu_\infty} [Q^{\alpha_{x_0}}(r^n_{x_0,h}) (h)] \stackrel{(\ref{1.6})}{\ge} e^{-\frac{h^2}{2}} E^{\nu_\infty} \big[r^n_{x_0,h}(h) \,P^Y [\alpha_{x_0}\, h + \sqrt{\alpha}_{x_0} \;Y \ge h]\big]
\\[1ex]
& \ge e^{-\frac{h^2}{2}} E^{\nu_\infty}  \big[d_{x_0} \le M, r^{n+1}_{x_0,h}(h) \;\ov{\Phi} \big((\mbox{\f $\dis\frac{1}{\sqrt{\alpha_{x_0}}}$} - \sqrt{\alpha}_{x_0}) \,h\big)\big] > 0,
\end{split}
\end{equation}
with $\ov{\Phi}(t) = P^Y[Y > t]$, for $t \in \IR$, and we have used that $r^{n+1}_{x_0,h} > 0$ on $[h,\infty)$.

\bigskip
When $\{ d_{x_0} \le M \}$, we have $\alpha^{-{1/2}}_{x_0} \le c(M)$ by (\ref{1.10}), (\ref{1.5}). Note also that $\ov{\Phi}(0) = 1/2 > 1/m'$. We can thus find by the definition of $\wt{\gamma}^n_h $ and  (\ref{5.60}) a small $h > 0$, such that
\begin{equation}\label{5.63}
\mbox{$\gamma^n_h  > \mbox{\f $\dis\frac{1}{m'}$} \; \wt{\gamma}^n_h $ , for all $n \ge 0$, as well as $\cL(\mbox{\f $\dis\frac{h^2}{2}$}) \ge \mbox{\f $\dis\frac{1 + s_0}{2}$}$}.
\end{equation}

\n
Since $f'_{\infty,M}(s) > m' $ for $ s_0 \le s \le 1$ it follows from (\ref{5.59}), (\ref{5.63}) that $\gamma^n_h \ge \frac{1-s_0}{2}$, for all $n \ge 0$. Hence, by monotone convergence, we find that
\begin{equation}\label{5.64}
E^{\nu_\infty} \big[e^{-\frac{h^2}{2(1 + R^\infty_{x_0})}} \;Q^{\alpha_{x_0}} (r_{x_0,h})(h)] \ge \mbox{\f $\dis\frac{1-s_0}{2}$} > 0.
\end{equation}

\n
This proves that with positive $P^{\nu_\infty}$-measure, $q_{x_0,h}$ is not identically $1$, and by the $0$-$1$ law stated above (\ref{5.5}) this happens $P^{\nu_\infty}$-almost surely. By the comment below (\ref{5.6}) we see that that $h_* \ge h$. This proves (\ref{5.56}), and concludes the proof of \hbox{Theorem \ref{theo5.8}.}
\end{proof}

\begin{remark}\label{rem5.9} \rm  Incidentally, Proposition \ref{prop4.2} implies that $h_* > 0$ for a binary branching, a case corresponding to $m=2$, which is not covered by Theorem \ref{theo5.8}. One can thus wonder about the nature of broader assumptions under which Theorem \ref{theo5.8} continues to hold. For instance, does $h_* > 0$ hold almost surely on non-extinction as soon as $m > 1$ ~? 
\hfill $\square$
\end{remark}
\appendix

\section{Appendix}
\setcounter{equation}{0}

In this appendix we provide for the reader's convenience a proof along the argument of Theorem 2 of \cite{BricLeboMaes87} showing that $h_* \ge 0$ in the general set-up of transient weighted graphs.

\medskip
We consider a locally finite, connected, transient weighted graph, with vertex set $E$, and symmetric weights $c_{x,y} = c_{y,x} \ge 0$, which are positive exactly when $x \sim y$, i.e.~when $x$ and $y$ are neighbors. We denote by $g(x,y)$, $x,y \in E$, the Green function, by $(\varphi_x)_{x \in E}$ the canonical Gaussian free field, and by $\IP^G$ its distribution. The discrete time walk on the weighted graph when located in $x$ jumps to a neighbor $y$ with probability $c_{x,y} / \lambda_x$, where $\lambda_x = \sum_{x' \sim x} c_{x,x'}$. It is governed by the law $P_x$. We use otherwise similar notation as in Section 1.

\medskip
We consider a base point $x_0 \in E$. We say that $C \subseteq E$ is a contour surrounding $x_0$ (in the terminology of  \cite{BricLeboMaes87}), when there exists a finite connected set $K \subseteq E$ containing $x_0$ such that $C = \partial K$, or when $C = \{x_0\}$ and we set $\{x_0\} = \partial \phi$ by convention. Given a contour $C$ surrounding $x_0$, we write ${\rm Int} \,C = K$ for the unique finite connected set $K$ containing $x_0$, such that $C = \partial K$, when $C \not= \{x_0\}$, or $K = \phi$, when $C = \{x_0\}$ (when $C \not= \{x_0\}$, ${\rm Int} \,C$ is the connected component of $E \backslash C$ containing $x_0$).

\medskip
Given a finite family of contours $C_i, 1 \le i \le n$ surrounding $x_0$, we define the maximal contour via
\begin{equation}\label{A1}
\max\{C_1,\dots,C_n\} = \partial \Big(\textstyle{\bigcup\limits^n_{i=1}} {\rm Int}\,C_i\Big),
\end{equation}
and observe that
\begin{equation}\label{A2}
\max\{C_1,\dots,C_n\} \subseteq \textstyle{\bigcup\limits^n_{i=1}}  \,C_i .
\end{equation}

\n
We now consider a finite connected set $U \ni x_0$, $\Delta = \partial U$ and $\ov{U} = U \cup \Delta$. Given $h \in \IR$, we introduce the disconnection event
\begin{equation}\label{A3}
\mbox{$D^h_{x_0,U} = \{x_0$ is not connected to $\Delta$ by a path in $\ov{U}$ where $\varphi \ge h\}$}.
\end{equation}

\begin{lemma}\label{lemA1}
\begin{align}
D^h_{x_0,U} = \{\varphi, &\; \mbox{there is a contour $C$ surrounding $x_0$, with} \label{A4}
\\[-0.5ex]
&\;\mbox{${\rm Int}\, C \subseteq U$, where $\varphi < h\}$}. \nonumber
\end{align}
\end{lemma}

\begin{proof}
Denote by $\wt{D}$ the event on the right-hand side of (\ref{A4}). First note that $\wt{D} \subseteq D^h_{x_0,U}$. Indeed, if $C$ is a contour surrounding $x_0$, with ${\rm Int} \, C\subseteq U$ and where $\varphi < h$, any path in $\ov{U}$ from $x_0$ to $\Delta = \partial U$ will exit ${\rm Int} \, C$ at a point of $C$ where $\varphi < h$. Conversely, one has $D^h_{x_0,U} \subseteq \wt{D}$. Indeed, when $D^h_{x_0,U}$ occurs, the connected component of $\{\varphi \ge h\}$ containing $x_0$ is contained in $U$ and its outer boundary (understood as $\{x_0\}$ when this component is empty) yields a contour $C$ with ${\rm Int} \, C \subseteq U$ where $\varphi < h$. This proves the lemma.
\end{proof}

On the disconnection event $D^h_{x_0,U}$, we can thus define with (\ref{A1})
\begin{align}
C_{< h}^{\max} (U) = &\; \mbox{the maximal contour of the family of contours $C$}\label{A5}
\\[-0.5ex]
&\;\mbox{surrounding $x_0$ with ${\rm Int} \,C \subseteq U$, where $\varphi < h$}. \nonumber
\end{align}

\n
We recall the definition (\ref{0.3}) of the critical value $h_*$.

\begin{proposition}\label{propA2}
\begin{equation}\label{A6}
h_* \ge 0.
\end{equation}
\end{proposition}

\begin{proof}
We will show that for any $\varepsilon > 0$,
\begin{equation} \label{A7}
\sup\limits_U \IP^G[D^{-\varepsilon}_{x_0,U}] < 1
\end{equation}

\n
($U$ runs over the collection of finite connected sets containing $x_0$).

\medskip
This will imply that for any $\varepsilon > 0$, with positive $\IP^G$-probability the connected component of $x_0$ in $\{\varphi \ge -\varepsilon\}$ is infinite, and (\ref{A6}) will follow.

\medskip
We thus prove (\ref{A7}). For $C$ a contour surrounding $x_0$, with ${\rm Int} \,C \subseteq U$, we have by the Markov property of $\varphi$ under $\IP^G$ (see for instance Proposition 2.3 of \cite{Szni12e})
\begin{equation}\label{A8}
\begin{array}{l}
\mbox{$\varphi_{x_0} = h_C + \xi_C$ where $h_C = E_{x_0} [\varphi(X_{H_C})]$ and}
\\
\mbox{$\xi_C$ is $N(0,g_{\rm Int \,C} (x_0, x_0))$-distributed and independent of $\sigma_{({\rm Int} \,C)^c}$}
\end{array}
\end{equation}
(the notation is similar as in (\ref{1.7})).

\medskip
By Lemma \ref{lemA1} and (\ref{A5}), we see that for $h= - \varepsilon$ and $C^{\max}_{< -\varepsilon}$ a shorthand for $C_{< -\varepsilon}^{\max}(U)$
\begin{equation}\label{A9}
\begin{split}
D^{-\varepsilon}_{x_0,U}  =  \textstyle{\bigcup\limits_C} \{C_{< -\varepsilon}^{\max} = C\},& \;\mbox{where $C$ runs over the collection of}
\\[-2.8ex]
&\; \mbox{contours surrounding $x_0$, with ${\rm Int} \,C \subseteq U$}.
\end{split}
\end{equation}
We thus find that
\begin{equation}\label{A10}
0 = \IE^G[{\rm sign}(\varphi_{x_0})] \stackrel{(\ref{A9})}{=} \dsl_C \,\IE^G[{\rm sign}(\varphi_{x_0}), \;C_{< -\varepsilon}^{\max} = C] + \IE^G[{\rm sign}(\varphi_{x_0}), (D^{-\varepsilon}_{x_0,U})^c],
\end{equation}

\medskip\n
where $C$ runs over the same family as in (\ref{A9}).

\medskip
Note that the event $\{C^{\max}_{< -\varepsilon} = C\}$ is $\sigma_{({\rm Int} \,C)^c}$-measurable, and by (\ref{A8}) we find that
\begin{equation}\label{A11}
\begin{array}{l}
\IE^G[{\rm sign}(\varphi_{x_0}), \; C^{\max}_{< - \varepsilon} = C] = \IE^G[{\rm sign}(h_C + \xi_C), \; C^{\max}_{< - \varepsilon} = C] = 
\\[2ex]
\IE^G\Big[\Big(2 \Phi \Big(\mbox{\f $\dis\frac{h_C}{\sqrt{g_{{\rm Int} \,C}(x_0,x_0)}}$}\Big) -1 \Big), \; C^{\max}_{< - \varepsilon} = C\Big] \le 
\\[3ex]
-\Big(2 \Phi \Big(\mbox{\f $\dis\frac{\varepsilon}{\sqrt{g(x_0,x_0)}}$}\Big) -1 \Big)\; \IP^G[C^{\max}_{< - \varepsilon} = C],
\end{array}
\end{equation}

\medskip\n
where $\Phi(t) = P^Y[Y \le t]$, for $t \in \IR$, with $Y$ a $N(0,1)$-distributed variable, and we have used  that $h_C \le - \varepsilon$ on $\{C^{\max}_{< -\epsilon} = C\}$ and $g_{{\rm Int} \,C}(x_0,x_0) \le g(x_0,x_0)$ for the last inequality of (\ref{A11}).

\medskip
Coming back to (\ref{A10}), we see that
\begin{equation}\label{A12}
\begin{array}{l}
0 \le -\Big(2 \Phi \Big(\mbox{\f $\dis\frac{\varepsilon}{\sqrt{g(x_0,x_0)}}$}\Big) - 1\Big) \;\dsl_C \IP^G[C_{< -\varepsilon}^{\max} = C] + 1 - \IP^G[D^{-\varepsilon}_{x_0,U}]
\\[1ex]
\stackrel{(\ref{A9})}{=} - 2 \Phi  \Big(\mbox{\f $\dis\frac{\varepsilon}{\sqrt{g(x_0,x_0)}}$}\Big) \,\IP^G [D^{-\varepsilon}_{x_0,U}] + 1.
\end{array}
\end{equation}
This shows that
\begin{equation}\label{A13}
\IP^G [D^{-\varepsilon}_{x_0,U}] \le \Big(2 \Phi \Big(\mbox{\f $\dis\frac{\varepsilon}{\sqrt{g(x_0,x_0)}}$}\Big) \Big)^{-1} (< 1),
\end{equation}
and proves (\ref{A7}). Our claim (\ref{A6}) follows.
\end{proof}

\bigskip\n
{\bf Acknowledgements:} We thank an anonymous referee for pointing out reference \cite{GrimKest}, which led to a simplification of the proof of Proposition \ref{prop5.2}.

\end{document}